\documentclass[11pt]{amsart}
\usepackage{lmodern}
\usepackage{amsmath, amsthm, amssymb, amsfonts}
\usepackage[normalem]{ulem}
\usepackage{hyperref}
\usepackage{enumitem} 

\usepackage{verbatim} 
\usepackage{longtable}

\usepackage{mathtools}

\usepackage{tikz}
\usetikzlibrary{decorations.pathmorphing}
\tikzset{snake it/.style={decorate, decoration=snake}}

\usepackage{caption}

\usepackage{tikz-cd}
\usetikzlibrary{arrows}

\theoremstyle{plain}
\newtheorem{thm}{Theorem}[section]
\newtheorem{cor}[thm]{Corollary}
\newtheorem{lem}[thm]{Lemma}
\newtheorem{prop}[thm]{Proposition}

\theoremstyle{definition}
\newtheorem{defn}[thm]{Definition}

\theoremstyle{remark}
\newtheorem{rmk}[thm]{Remark}

\newcommand{\BA}{{\mathbb{A}}}

\newcommand{\BC}{{\mathbb{C}}}

\newcommand{\BH}{{\mathbb{H}}}

\newcommand{\BP}{{\mathbb{P}}}
\newcommand{\BQ}{{\mathbb{Q}}}

\newcommand{\BZ}{{\mathbb{Z}}}

\newcommand{\CC}{{\mathcal C}}

\newcommand{\CI}{{\mathcal I}}

\newcommand{\CN}{{\mathcal N}}
\newcommand{\CO}{{\mathcal O}}
\newcommand{\CP}{{\mathcal P}}
\newcommand{\CQ}{{\mathcal Q}}

\newcommand{\Fd}{{\mathfrak{d}}}

\newcommand{\Ff}{{\mathfrak{f}}}
\newcommand{\Fg}{{\mathfrak{g}}}
\newcommand{\Fh}{{\mathfrak{h}}}

\newcommand{\Fk}{{\mathfrak{k}}}

\newcommand{\Fn}{{\mathfrak{n}}}

\newcommand{\Fp}{{\mathfrak{p}}}
\newcommand{\Fq}{{\mathfrak{q}}}

\newcommand{\ch}{{\mathrm{ch}}}

\DeclareFontFamily{OT1}{rsfs}{}
\DeclareFontShape{OT1}{rsfs}{n}{it}{<-> rsfs10}{}
\DeclareMathAlphabet{\curly}{OT1}{rsfs}{n}{it}

\newcommand\Id{\operatorname{Id}}

\usepackage{tikz}
\usepackage{lmodern}
\usetikzlibrary{decorations.pathmorphing}

\title[Multiplicativity of perverse filtrations]{Multiplicativity of perverse filtration for Hilbert schemes of fibered surfaces, II}
\author{Zili Zhang}
\address{University of Michigan, Department of Mathematics}
\email{ziliz@umich.edu}
\date{\today}
\begin{document}
\begin{abstract}
Let $S\to C$ be a smooth quasi-projective surface properly fibered onto a smooth curve. We prove that the multiplicativity of the perverse filtration on $H^*(S^{[n]},\BQ)$ associated with the natural map $S^{[n]}\to C^{(n)}$ implies that $S\to C$ is an elliptic fibration. The converse is also true when $S\to C$ is a Hitchin-type elliptic fibration.
\end{abstract}

\maketitle
\tableofcontents

\section{Introduction}
Let $f:X\to Y$ be a proper morphism between smooth complex algebraic varieties. There is an increasing filtration,
\[
P_0H^d(X,\BQ)\subset P_1H^d(X,\BQ)\subset\cdots \subset H^d(X,\BQ),
\]
called the perverse filtration associated with $f$. It is defined by truncating the derived pushforward $Rf_*\BQ_X$ by the perverse $t$-structure on $D^b_c(Y)$, the bounded derived category of constructible sheaves on $Y$. See Section 4.1 for detailed discussions. The perverse filtration is said to be {\it multiplicative} if the cup product satisfies
\[
P_kH^d(X,\BQ)\times P_{k'}H^{d'}(X,\BQ)\xrightarrow{\cup} P_{k+k'}H^{d+d'}(X,\BQ)
\]
for any $k,k',d,d'\ge0$. The purpose of this paper is to study the multiplicativity of the perverse filtration associated with the natural projection from Hilbert schemes of points on fibered surfaces to certain naturally defined bases.

\subsection{Motivation from the P=W conjecture}
Our motivation of studying the multiplicativity of perverse filtration arises from the $P=W$ conjecture. Let $C$ be a smooth projective curve of genus at least $2$. There are two moduli spaces which are attached to the curve $C$ and an integer $n$. They are Simpson's Dolbeault and Betti moduli spaces. The Doubeault moduli space $M_D$ parametrizes degree $0$ stable Higgs bundles of rank $n$ on $C$, and the Betti moduli space $M_B$ is the corresponding character variety. In \cite{S}, Simpson constructed a diffeomorphism between $M_D$ and $M_B$, called the nonabelian Hodge correspondence or the Simpson's correspondence. A remarkable prediction, suggested by de Cataldo, Hausel, and Migliorini asserts that under the identification 
\[
H^*(M_D,\BQ)=H^*(M_B,\BQ),
\]
induced by the pullback of the Simpson's diffeomorphism, the perverse filtration on $M_D$ associated with the Hitchin fibration matches the halved Hodge-theoretic weight filtration on $M_B$, i.e.
\[
P_kH^*(M_D,\BQ)=W_{2k}H^*(M_B,\BQ)=W_{2k+1}H^*(M_B,\BQ),~~k\ge0.
\]
Such a phenomenon is referred to as the ``$P=W$ conjecture''. It was proved in \cite{dCHM} in the case of $n=2$ and $g\ge 2$, and recently in \cite{dCMS} for arbitrary $n$ and $g=2$. Since the Hodge-theoretic weight filtration is always multiplicative, the multiplicativity of the perverse filtration associated with the Hitchin map is strong evidence to support the $P=W$ conjecture. In fact, it is proved that the $P=W$ conjecture is equivalent to the existence of such a multiplicativity, \cite[Theorem 0.6]{dCMS}.

There is a parabolic version of the Simpson's correspondence and hence a parabolic version of the $P=W$ conjecture. The parabolic $P=W$ conjecture is proved for five families of parabolic moduli spaces indexed by affine Dynkin diagrams and the rank $n$ in \cite{SZ} and \cite{Z}. In this setting, each Dolbeault moduli space $M_D$ is of the form  $S^{[n]}$, a Hilbert scheme of $n$ points on smooth elliptically fibered surface $f:S\to\BA^1$. The corresponding Hitchin maps $\pi_n:S^{[n]}\to\BA^n$ are constructed as the composition of the Hilbert-Chow morphism $S^{[n]}\to S^{(n)}$ and the natural projection $S^{(n)}\to (\BA^1)^{(n)}$, where $S^{(n)}$ denotes the $n$-th symmetric product of $S$. In fact, the multiplicativity of perverse filtration argument in the proof works in the following generality.

\begin{thm}{\cite[Theorem 4.18]{Z}}
Let $f:S\to C$ be a surjective morphism from a smooth projective surface with numerically trivial canonical bundle to a smooth projective curve. Then the perverse filtration associated with the morphism $S^{[n]}\to C^{(n)}$ is multiplicative. 
\end{thm}

It is natural to ask whether the same result holds for surfaces with non-trivial canonical bundle. In this paper, we will show how the multiplicativity of the perverse decomposition on the Hilbert scheme $H^*(S^{[n]},\BQ)$ is governed by the geometry of the fibered surface.

\subsection{Hilbert schemes of Hitchin-type fibrations}
Let $f:S\to \BA^1$ be a proper surjective morphism from a connected smooth quasi-projective surface to the affine line, such that the restriction map 
\[
H^*(S,\BQ)\to H^*(f^{-1}(0),\BQ)
\]
is an isomorphism. Such a morphism is called a \emph{Hitchin-type fibration}. The perverse filtration associated with  $f$ is a 2-step filtration
\[
P_0H^*(S,\BQ)\subset P_1H^*(S,\BQ)\subset P_2H^*(S,\BQ)=H^*(S,\BQ),
\]
which is always multiplicative, \cite[Proposition 4.17]{Z}. A decomposition $G_\bullet H^*(S,\BQ)$ which splits the filtration $P_\bullet H^*(S,\BQ)$, \emph{i.e.}
\[
P_kH^*(S,\BQ)=\bigoplus_{i=0}^k G_iH^*(S,\BQ),~~k=0,1,2,
\]
canonically induces a decomposition
\[
H^*(S^{[n]},\BQ)=\bigoplus_{i=0}^{2n}G_iH^*(S^{[n]},\BQ)
\]
which splits the perverse filtration associated with the natural map $\pi_n:S^{[n]}\to \BA^n$. A splitting $G_\bullet H^*(S^{[n]},\BQ)$ is called \emph{strongly multiplicative} if
\[
G_kH^d(S^{[n]},\BQ)\times G_{k'}H^{d'}(S^{[n]},\BQ)\xrightarrow{\cup}G_{k+k'}H^{d+d'}(S^{[n]},\BQ)
\]
for any $d,d',k,k'\ge0$. We see from the definition that the perverse filtration associated with $\pi_n$ is multiplicative if it admits a strongly multiplicative splitting. For a given perverse filtration,  we are free to choose a good splitting adapted to the problem we would like to solve. This idea is frequently used in the study of perverse filtrations related to the $P=W$ conjecture; see \cite[Section 3]{dCMS}, \cite[Section 3]{SZ}. Our main theorem is:

\begin{thm}[Theorem \ref{4.5*}]\label{0.0}
Let $n\ge2$. Let $f:S\to\BA^1$ be a Hitchin-type fibration. Let $\pi_n:S^{[n]}\to \BA^n$ be the induced map. Then the following are equivalent.
\begin{enumerate}
    \item The perverse filtration associated with $\pi_n$ is multiplicative.
    \item The perverse filtration associated with $\pi_n$ admits a strongly multiplicative splitting.
    \item The canonical class $K_S\in P_1H^2(S,\BQ)$.
    \item $K_S\in G_1H^2(S,\BQ)$ in any splitting of $P_\bullet H^*(S,\BQ)$.
    \item $f:S\to \BA^1$ is an elliptic fibration.
\end{enumerate}
\end{thm}

When $n=1$, the equivalent conditions break into three groups. Condition (1) is always true, \cite[Proposition 4.17]{Z}. For condition (2), we find the following topological condition to detect the existence of a strongly multiplicative splitting.

\begin{thm}[Theorem \ref{star}]
Let $f:S\to \BA^1$ be a Hitchin-type fibration. Then the perverse filtration associated with $f$ admits a strongly multiplicative splitting if and only if  $f^{-1}(0)$ has at most 1 irreducible components of positive geometric genus.
\end{thm}

The conditions (3), (4), and (5) are still equivalent for $n=1$. Moreover, the equivalence between (3) and (5) holds for general fibered surfaces.

\begin{prop}[Proposition \ref{pf2}]
Let $f:S\to C$ be a proper surjective map from a quasi-projective surface to a smooth curve. Then $K_S\in P_1H^2(S,\BQ)$ if and only if $f$ is an elliptic fibration.
\end{prop}

For general fibered surfaces $f:S\to C$ and $n\ge 2$, we have the following necessary condition for the multiplicativity of the perverse filtration associated with $\pi_n$, \emph{i.e.} (1)$\Rightarrow$(5), which is also sufficient for $n=2$.

\begin{thm}[Theorem \ref{pie}, \ref{n=2}]
Let $f:S\to C$ be a proper surjective morphism from a smooth quasi-projective to a smooth curve. Let $\pi_n:S^{[n]}\to C^{(n)}$ be the induced morphism. Then
\begin{enumerate}
    \item The perverse filtration associated with $\pi_1$ is always multiplicative.
    \item Let $n\ge 2$. If the perverse filtration associated with $\pi_n$ is multiplicative, then $f$ is an elliptic fibration. The converse is true for $n=2$.
\end{enumerate}
\end{thm}

For $n=1$ and $n=2$, we perform a direct computation using the explicit geometry of the fibrations and the intersection theory to study the perverse filtrations and their splittings. The canonical class $K_S$ enters the picture when we study the self-intersection of the boundary class $\partial S^{[2]}$, which does not appear when $n=1$. For general $n$, we develop a systematic framework called $G$-decomposition to study the cohomology ring $H^*(S^{[n]},\BQ)$. The $G$-decomposition captures certain important properties of the perverse decompositions associated with Hilbert schemes of points developed in \cite{SZ} and \cite{Z}. Besides the perverse filtrations, the graded pieces of the Hodge-theoretic weight filtrations under mild conditions or the Hodge filtration for smooth projective varieties also form $G$-decompositions.

\subsection{G-decomposition}
Let $X$ be a smooth quasi-projective variety of dimension $n$. Let $F$ be a field of characteristic 0. A \emph{G-decomposition of length $m$ with coefficient $F$} is a pair of decompositions 
\[
H^*(X,F)=\bigoplus_{i=0}^mG_i H^*(X,F) \textrm{ and } H^*_c(X,F)=\bigoplus_{i=0}^mG'_i H^*_c(X,F)
\] 
such that the following properties hold.
\begin{enumerate}
    \item $1\in G_0H^0(X,F)$.
    \item The forgetful map $H^*_c(X,F)\to H^*(X,F)$ is compatible with the decompositions $G$ and $G'$, \emph{i.e.} $G'_kH^*_c(X,F)\to G_kH^*(X,F)$, and
    \item The decompositions $G$ and $G'$ are dual to each other with respect to the Poincar\'e paring, \emph{i.e.} $G_kH^d(X,F)\times G'_{k'}H^{d'}_c(X,F)\to F$ is a perfect paring if $k+k'=m$ and $d+d'=2n$, and is 0 otherwise.
\end{enumerate}

Since $G'$ is uniquely determined by $G$ by (2), we will simply say a $G$-decomposition $G_\bullet H^*(X,F)$ without mentioning $G'$ when no confusion arises. For the purpose of this paper, we fix the length $m=\dim X$ and the coefficient $F=\BQ$ of $G$-decompositions. Note that when $S$ is a smooth quasi-projective surface, a $G$-decomposition on $H^*(S,\BQ)$ canonically induces a $G$-decomposition on $H^*(S^{[n]},\BQ)$. We have:

\begin{thm}[Theorem \ref{6.1}, \ref{6.2}]\label{1.2}
Let $n\ge2$. Let $S$ be a smooth quasi-projective surface equipped with a $G$-decomposition  $G_\bullet H^*(S,\BQ)$. If the induced $G$-decomposition on $H^*(S^{[n]},\BQ)$ is strongly multiplicative, then $K\in G_1H^2(S,\BQ)$. Suppose that the $G$-decomposition on $H^*(S,\BQ)$ is strongly multiplicative, then the converse is also true.
\end{thm}

Our approach to prove Theorem \ref{1.2} is to understand how the Heisenberg algebra and Virasoro algebra act on the $G$-decompositions. We show that they are pure with respect to the $G$-decompositions, i.e. the image of any direct summand is contained in a single direct summand. As a consequence, we are able to determine the direct summands where the tautological classes locate in the $G$-decomposition. We list the properties of various operators here, following the notations in Section 2: 

\begin{thm}[Proposition \ref{4.2}, \ref{4.3}, \ref{4.7}, \ref{5.1}]\label{1.3}
Let $S$ be a quasi-projective surface equipped with a G-decomposition $G_\bullet H^*(S,\BQ)$. Then $G_\bullet$ induces a $G$-decomposition on $H^*(S^{[n]},\BQ)$ for all $n\ge 1$.
Therefore, the Fock space $\BH=\oplus_nH^*(S^{[n]},\BQ)$ is a trigraded vector space
\[
\BH=\bigoplus_{n,d,k} G_kH^d(S^{[n]},\BQ)
\]
with grading $(n,d,k)$. Let $\alpha\in G_kH^d(S,\BQ)$ if $n\ge 0$, and $G'_kH^d_c(S,\BQ)$ if $n<0$. We have
\begin{enumerate}
    \item The Nakajima operator $\Fq_n(\alpha)$ is of degree $(n,d+2n-2,k+n-1)$.
    \item The Virasoro operator $L_n(1)$ is of degree $(n,2n,n)$.
\end{enumerate}
Suppose that $G_\bullet H^*(S,\BQ)$ is strongly multiplicative.
\begin{enumerate}
    \item[(3)] The Virasoro operator $L_n(\alpha)$ is of degree $(n,d+2n,k+n)$. 
\end{enumerate}
Suppose further that $K_S\in G_1H^2(S,\BQ)$.
\begin{enumerate}
    \item[(4)] The boundary operator $\partial$ is of degree $(0,2,1)$.
    \item[(5)] The ``cupping with $\alpha^{[n]}_l$" operator is of degree $(0,d+2l-4,k+l-2)$. 
\end{enumerate}
\end{thm}

\subsection{Outline} The paper is organized as follows. In Section 2, we set up the notations for Hilbert schemes and recall the work of Lehn and Li-Qin-Wang on tautological classes and cup products of Hilbert schemes of points on smooth surfaces. We also define $G$-decompositions as a generalization of perverse decomposition, and recall some basic facts about $G$-decompositions on the cohomology of Hilbert schemes. In Section 3, we calculate the $G$-degree of Nakajima operators, Virasoro operators, the boundary operator, and ``cupping with tautological class'' operators. As a corollary, we determine the precise $G$-degree of tautological classes. We prove our main theorem by using Lehn's formula to calculate the $G$-degrees of the cup product. Section 4 contains our main application on multiplicativity of perverse decomposition.

\subsection{Acknowledgements}
I thank Mark de Cataldo, Shizhang Li, and Junliang Shen for helpful discussions.

\section{Hilbert schemes}
\subsection{Cup product}
In this section we set up notations and recall cup product formula following \cite{GS,L,LQW}. Let $X$ be a smooth quasi-projective surface. Let $X^{[n]}$ be the Hilbert scheme of $n$ points on $X$. We denote
\[
\BH=\bigoplus_{n=0}^\infty H^*(X^{[n]},\BQ).
\]
There is a distinguished element in $H^*(X^{[0]},\BQ)=\BQ$, denoted by $\mathbf{1}$.  Let $Z_{n,n+k}\subset X\times X^{[n]}\times X^{[n+k]}$ be the subvariety defined by
\[
Z_{n,n+k}=\left\{(x,\xi,\eta)\mid \CI_\eta\subset\CI_\xi,\,\CI_\xi/\CI_\eta\textrm{ is supported at }x\right\}.
\]
Let $p_1,p_2,p_3$ be the projections to the three factors. For $\alpha\in H^*(X)$, $k>0$ we define 
\begin{align*}
\Fq_k(\alpha): H^*(X^{[n]})&\to H^*(X^{[n+k]})\\
               \gamma&\mapsto p_{3*}(p_1^*\alpha\cdot p_2^*\gamma\cdot Z_{n,k+n})
\end{align*}
For $\beta\in H^*_c(X)$ and $k>0$, we define
\begin{align*}
\Fq_{-k}(\beta): H^*(X^{[n+k]})&\to H^*(X^{[n]})\\
            \gamma&\mapsto p_{2*}(p_1^*\beta\cdot p_3^*\gamma\cdot Z_{k,k+n}).
\end{align*}

When $k=0$, the Nakajima operators is identically 0. The Nakajima operators $\Fq_{k}$ satisfy the following identity.

\begin{thm}{\cite[Theorem 3.1]{N}}\label{nak}
For $n_1,n_2\in\BZ$, $\alpha_i\in H^*(X,\BQ)$ if $n_i>0$, and $\alpha_i\in H_c^*(X,\BQ)$ if $n_i<0$, then the following relation hold.
\[
[\Fq_{n_1}(\alpha_1),\Fq_{n_2}(\alpha_2)]=n\delta_{n_1,-n_2}\int_X{\alpha_1\alpha_2}\cdot\Id_\BH,
\]
where $[\cdot,\cdot]$ is the supercommutator and the $\delta$ is the Kronecker function.
\end{thm}
G\"ottsche's formula represents the (bi)graded ring $\BH$ as a symmetric product as follows.
\begin{thm}{\cite[Theorem 3.6]{LS}}
There is an isomorphism between graded vector spaces
\begin{align*}
\textnormal{Sym}^*(H^*(X,\BQ)\otimes t\BQ[t])&\cong\bigoplus_{n\ge0}H^*(X^{[n]},\BQ), \\
(\alpha_1t^{n_1})\cdots (\alpha_st^{n_s})&\mapsto \Fq_{n_1}(\alpha_1)\cdots\Fq_{n_s}(\alpha_s)\mathbf{1}.
\end{align*}
\end{thm}

In particular, fix a linear basis $B$ of $H^*(X,\BQ)$, the set
\[
\{\Fq_{n_1}(\alpha_1)\cdots\Fq_{n_s}(\alpha_s)\mathbf{1}\mid n_1+\cdots+n_s=n,\,\alpha_i\,\in B,\,n_i>0,\,1\le i\le s\}
\]
is a linear basis of $H^*(X^{[n]},\BQ)$. Let $Z_n\subset X^{[n]}\times X$ be the universal subscheme, and let $p:Z_n\to X^{[n]}$ and $q:Z_n\to X$ be the natural projections. For any element $\alpha\in H^*(X,\BQ)$, denote
\[
\alpha^{[n]}=p_*(\textrm{ch}(\CO_{Z_n})\cdot q^*(\alpha\cdot \textrm{td}(S)))
\]
and
\begin{equation}\label{taut}
\begin{split}
\alpha ^{[n]}_k=&p_*(\ch_k(\CO_{Z_n})\cdot q^*(\alpha\cdot \textrm{td}_0(S))\\
&+\ch_{k-1}(\CO_{Z_n})\cdot q^*(\alpha\cdot \textrm{td}_1(S))\\
&+\ch_{k-2}(\CO_{Z_n})\cdot q^*(\alpha\cdot \textrm{td}_2(S))).
\end{split}
\end{equation}
We have
\[
\alpha^{[n]}=\sum_{k\ge0}\alpha^{[n]}_k.
\]
It follows from the definition that $\deg \alpha^{[n]}_k=\deg\alpha+2k-4$. We denote by $\alpha^{[\bullet]}\in \textrm{End}_\BQ{\BH}$ the linear operator which is multiplication by $\alpha^{[n]}$ on $H^*(X^{[n]},\BQ)$. We denote the homogeneous degree 2 component of $1^{[\bullet]}$ as $\partial$. 

\begin{thm}{\cite[Theorem 4.2]{L}} \label{pr}
For $\alpha,y\in H^*(X,\BQ)$.
\begin{equation}\label{pr1}
[\alpha^{[\bullet]},\Fq_{1}(y)]=\textnormal{exp}(\textnormal{ad}\,\partial)\Fq_1(\alpha\cdot y),
\end{equation}
where $\textup{ad}(\partial)(-):=[\partial,-]$.
\end{thm}

We recall the definition of the Virasoro operators $L_n(\alpha)$ for quasi-projective surfaces following \cite{N1}. Let $\Delta_+$ and $\Delta_-$ be the two ``$\otimes$-Hom'' adjoints of the cup product 
\[
\cup:H^*_c(X,\BQ)\otimes H^*(X,\BQ)\to H^*_c(X,\BQ)
\]
with respect to the first and the second factor. Concretely, if $\{\beta_i\}$ is a basis of $H^*(X,\BQ)$ with dual basis $\{\beta^i\}$ in $H^*_c(X,\BQ)$ with respect to the Poincar\'e pairing, then
\begin{align}
\Delta_+: H^*(X,\BQ)&\to H^*(X,\BQ)\otimes H^*_c(X,\BQ)\label{Delta+}\\
\alpha&\mapsto \sum_i\beta_i\otimes\beta^i\alpha,\nonumber
\end{align}
and
\begin{align}
\Delta_-: H^*_c(X,\BQ)&\to H^*_c(X,\BQ)\otimes H^*_c(X,\BQ)\label{Delta-}\\
\alpha&\mapsto \sum_i\beta^i\otimes\beta_i\alpha.\nonumber
\end{align}
For integer $m,n\in\BZ$ and 
$\alpha,\beta\in H^*(X,\BQ)$ or $H^*_c(X,\BQ)$, formally denote
\[
\Fq_{m}\Fq_{n}(\alpha\otimes\beta):=\Fq_{m}(\alpha)\Fq_{n}(\beta)
\footnote{When $m>0$ and $\alpha\in H^*_c(X,\BQ)$, the operator $\Fq_{m}(\alpha)$ is understood as $\Fq_{m}(\iota\alpha)$ where $\iota:H^*_c(X,\BQ)\to H^*(X,\BQ)$ is the forgetful map. Similar for $n$ and $\beta$.}
\]
and extend it linearly by
\[
\Fq_m\Fq_n(\sum\alpha_i\otimes\beta_i):=\sum \Fq_m\Fq_n(\alpha_i\otimes\beta_i).
\]
For $\alpha\in H^*(X,\BQ)$ if $n\ge0$, and $\alpha\in H^*_c(X,\BQ)$ if $n<0$, define the Virasoro operators as
\[
L_n(\alpha)=
\begin{cases}
\frac{1}{2}\sum_{k\in\BZ}:\Fq_k\Fq_{n-k}:\Delta_+(\alpha), &n\ge0\\
\frac{1}{2}\sum_{k\in\BZ}:\Fq_k\Fq_{n-k}:\Delta_-(\alpha), &n<0\\
\end{cases}
\]
where $:-:$ denotes the normal ordered product, (\emph{i.e.} $:\Fq_k\Fq_{n-k}:=\Fq_k\Fq_{n-k}$ if $k\ge n-k$, and $\Fq_{n-k}\Fq_k$ otherwise).

The interactions between Nakajima operators and Virasoro operators are
\begin{equation} \label{lp}
[L_m(\beta),\Fq_n(\alpha)]=n\Fq_{m+n}(\beta\alpha).
\end{equation}

The following theorem was first proved in \cite{L} for the projective case and then generalized to the quasi-projective case in \cite{N1}.  

\begin{thm}{\cite[Corollary 4.3]{N1}} \label{2.4}
For $n\in\BZ$ and $\alpha\in H^*(X,\BQ)$, one has
\begin{equation}\label{partialp}
[\partial,\Fq_n(\alpha)]=nL_n(\alpha)+\binom{n}{2}\Fq_n(K_X\alpha),
\end{equation}
where $K_X$ is the canonical divisor class of $X$.
\end{thm}

Let $n=1$ and $\beta=1$ in (\ref{lp}), one may represent $\Fq_n(\alpha)$ by $L_1(1)$ and $\Fq_{n-1}(\alpha)$. Equation (\ref{partialp}) replaces $L_1(1)$ by $[\partial,\Fq_1(1)]$. Iterating this process leads to an identity
\[
\frac{1}{n!}\textrm{ad}^n([\partial,\Fq_1(1)])(\Fq_1(\alpha))=\Fq_{n+1}(\alpha).
\]
Therefore, we have
\[
\BH=\partial\BH+\sum_\alpha\Fq_1(\alpha)\BH,
\]
where $\alpha$ runs through a linear basis of $H^*(X,\BQ)$.

The cohomology ring $H^*(X^{[n]},\BQ)$ is generated by the tautological classes $\alpha^{[n]}_k$ when $\alpha$ runs over a linear basis and $k\ge 0$, \cite[Theorem 5.31]{LQW}. Therefore, to describe the cup product in $H^*(X^{[n]},\BQ)$ it suffices to interpret $\alpha^{[n]}_k\cdot \Fq_{n_1}(\alpha_1)\cdots\Fq_{n_s}(\alpha_s)\mathbf{1}$ in terms of Nakajima operators and Virasoro operators. Since the relations among various operators are given in terms of their commutators, we will iterate the following lemma later.

\begin{lem}\label{comm}
Let $\mathfrak{F}$ be a superalgebra acting on a vector space $V$ with supercommutator  
\[
[\Ff,\Fh]=\Ff\Fh-(-1)^{\deg\Ff\cdot\deg\Fh}\Fh\Ff.
\]
Let $\Ff_1,\cdots,\Ff_s,\Fh\in\mathfrak{F}$. Then we have an equality of operators
\begin{equation} \label{c}
\begin{split}
\Fh\Ff_1\cdots \Ff_s=&\sum_{i=1}^s(-1)^{c_i}\Ff_1\cdots \Ff_{i-1}[\Fh,\Ff_i]\Ff_{i+1}\cdots \Ff_s\\ 
&+(-1)^c\Ff_1\cdots \Ff_s\Fh,
\end{split}
\end{equation}
where $c_i=\deg \Fh\cdot(\deg \Ff_1+\cdots +\deg \Ff_{i-1})$ and $c=\deg \Fh\cdot(\deg \Ff_1+\cdots +\deg \Ff_{s})$ are two constants.
\end{lem}

\begin{rmk}
In applications, the sign before each term in (\ref{c}) is irrelevant. Instead of keeping track with the complicated constants $c_i$ and $c$, we will simply write
\begin{equation*}
\Fh\Ff_1\cdots \Ff_s=\pm\Ff_1\cdots \Ff_s\Fh+\sum_{i=1}^s\pm\Ff_1\cdots \Ff_{i-1}[\Fh,\Ff_i]\Ff_{i+1}\cdots \Ff_s.
\end{equation*}
\end{rmk}

\subsection{G-decomposition for Hilbert schemes}
In this section, we introduce the notion of $G$-decomposition on cohomology groups of smooth quasi-projective varieties, generalizing perverse decompositions and Hodge decompositions. 

\begin{defn} \label{du}
Let $F$ be a field of characteristic 0. Let $V$ be a finite $F$-algebra. Let
\begin{equation}\label{du1}
V=\bigoplus_{k=0}^{m} \bigoplus_{d=0}^{2n}V_k^d
\end{equation}
be a direct sum decomposition.
\begin{enumerate}[label=(\roman*)]
    \item A vector $v\in V$ is called \emph{pure} with respect to the decomposition \ref{du1} if $v\in V_k^d$ for some $k$ and $d$.
    \item A nonzero pure vector $v\in V$ is of \emph{bidegree} $(d,k)$ if $v\in V_k^d$.
    \item A basis $\beta_i\in V$ is \emph{adapted to the decomposition} (\ref{du1}) if
    \[
    V_k^d=\langle\beta_i\mid \beta_i\in V_k^d\rangle.
    \]
    \item The decomposition (\ref{du1}) is called \emph{strongly multiplicative} if the multiplication satisfies
    \[
    V_k^d\times V_{k'}^{d'}\xrightarrow{\cdot} V_{k+k'}^{d+d'}.
    \]
    \item Let $B:V\times V'\to F$ be a non-degenerate pairing of finite-dimensional $F$-vector spaces. Let 
\begin{equation}\label{9}
V=\bigoplus_{k=0}^{m}\bigoplus_{d=0}^{2n} V_i^d \textrm{ and } W=\bigoplus_{k'=0}^m\bigoplus_{d'=0}^{2n} W_{i'}^{d'}    
\end{equation}
be two decompositions. They are called \emph{dual decompositions with respect to the paring $B$} if the induced pairing on $V_k^d\times W_{k'}^{d'}$ is non-degenerate if $k+k'=m$ and $d+d'=2n$, and is 0 otherwise. 
\end{enumerate}

\end{defn}

\begin{defn}\label{dfn}
Let $X$ be a connected smooth quasi-projective variety of dimension $n$. Let $F$ be a field of characteristic 0.
\begin{enumerate}
\item A \emph{G-decomposition of length $m$ with coefficient $F$} is a pair of dual decompositions 
\begin{align*}
H^*(X,F)=&\bigoplus_{k=0}^m\bigoplus_{d=0}^{2n}G_k H^d(X,F) \\ H^*_c(X,F)=&\bigoplus_{k=0}^m\bigoplus_{d=0}^{2n}G'_{k} H^{d}_c(X,F)
\end{align*}
with respect to the Poincar\'e paring in the sense of Definition \ref{du}.(5) with the properties that
(i) $1\in G_0H^0(X,F)$ and (ii) the forgetful map $\iota:H^*_c(X,F)\to H^*(X,F)$ is compatible with the decompositions $G$ and $G'$, \emph{i.e.} 
    \[\iota:G'_kH^d_c(X,F)\to G_kH^d(X,F).\]  

\item A non-zero pure class $\alpha\in H^*(X,F)$ or $H^*_c(X,F)$ is of {\it $G$-degree $k$}, denoted as $\Fg(\alpha)=k$, if $\alpha\in G_kH^d(X,F)$ or $ G_kH^d_c(X,F)$ for some $d$. 

\end{enumerate}
\end{defn}

\begin{rmk} \label{G}
Since the decomposition $G'$ in Definition \ref{dfn} is uniquely determined by $G$, we will simply say a $G$-decomposition $G_\bullet H^*(X,F)$ without mentioning $G'$ when no confusing arises. The concept of the  $G$-decomposition can be also defined concretely as follows. Let $G_\bullet H^*(X,F)$ and $G'_\bullet H^*_c(X,F)$ be two direct sum decompositions such that $1\in G_0H^0(X,F)$. They form a $G$-decomposition if and only if there exists a basis $\{\beta_i\}$ adapted to the $G_\bullet H^*(X,F)$  and a basis $\{\beta^i\}$ adapted to $G'_\bullet H^*_c(X,F)$ such that
\begin{enumerate}
    \item $\int_X\beta_i\cdot\beta^j=\delta_{i,j},$
    \item $\iota\beta^i$ are pure for all $i$, and
    \item $\Fg(\iota\beta^i)=\Fg(\beta^i)=m-\Fg(\beta_i)$ whenever $\iota\beta^i\neq0$. 
\end{enumerate}
\end{rmk}

\begin{rmk}
The motivation and main examples of $G$-decomposition are the following.
\begin{enumerate}
    \item Let $f:X\to Y$ be a proper flat morphism of relative dimension $r$ between smooth varieties. Then $H^*(X,\BQ)$ and $H^*_c(X,\BQ)$ are equipped with perverse filtrations, whose graded pieces form a $G$-decomposition of length $m=2r$ with coefficient $\BQ$. See Section 4.1 for details. 
    \item Let $X$ be a smooth projective variety of dimension $n$. Then $H^*(X)$ and $H^*_c(X)$ are naturally identified. The Hodge decomposition 
    \[
    G_pH^*(X,\BC)=\bigoplus_{q=0}^n H^{p,q}_{\bar{\partial}}(X)
    \]
    is a $G$-decomposition of length $m=n$ with coefficient $\BC$. Similarly, the conjugate Hodge filtration is also a $G$-decomposition of length $n$ with coefficient $\BC$.
    \item Let $X$ be a complex algebraic variety of dimension $n$ such that the weight filtration in the mixed Hodge structure splits over field $F$. Suppose the 1-dimensional space $H^{2n}_c(X,F)$ is of pure weight $m$. Since the Poincar\'e pairing is a map of mixed Hodge structures, the splitting of the weight filtration is a $G$-decomposition of length $m$ with coefficient $F$.
\end{enumerate}
\end{rmk}

\noindent
{\bf Convention.} 
In this paper, fix the length $m=\dim X$ and coefficient $\BQ$.

\medskip

Let $^1G_\bullet H^*(X,\BQ)$ and $^2G_\bullet H^*(Y,\BQ)$ be two $G$-decompositions. Then we define the $G$-decomposition on the the product $X\times Y$ following the K\"unneth formula
\[
G_kH^*(X\times Y, \BQ)=\sum_{i+j=k} {^1G}_i H^*(X,\BQ)\otimes {^2G}_j H^*(Y,\BQ).\\
\]
It is straightforward to check the compatibility with the compactly supported cohomology.

Let $S$ be a smooth quasi-projective surface. Then a $G$-decomposition on $H^*(S,\BQ)$ canonically induces $G$-decompositions on the Cartesian powers $H^*(S^n,\BQ)$, the symmetric power $H^*(S^{(n)},\BQ)$, and the Hilbert scheme $H^*(S^{[n]},\BQ)$. By abuse of notation, we denote all of them by $G$. On the Cartesian power $S^n$, we define 
\[
G_kH^*(S^n,\BQ)=\langle\alpha_1\boxtimes\cdots\boxtimes\alpha_n\mid \alpha_i\in G_{k_i}H^*(S,\BQ),\sum{k_i}=k\rangle.
\]
By taking the $\mathfrak{S}_n$-invariant part, the decomposition descend to the one for the symmetric product $S^{(n)}$,
\[
G_kH^*(S^{(n)},\BQ)=\langle \textrm{P}(\alpha_1\boxtimes\cdots\boxtimes\alpha_n)\mid \alpha_i\in G_{k_i}H^*(S,\BQ),\sum{k_i}=k\rangle.
\]
The operator $\textrm{P}$ is the symmetrization operator
\begin{equation}\label{P}
\textrm{P}(\alpha_1\boxtimes\cdots\boxtimes\alpha_n)=\frac{1}{n!}\sum_{\sigma\in S_n}(-1)^{\nu(\sigma,\alpha_\bullet)} \alpha_{\sigma(1)}\boxtimes\cdots\boxtimes\alpha_{\sigma(n)},
\end{equation}
where $\nu(\sigma,\alpha_\bullet)=\sum_{i<j,\sigma(i)>\sigma(j)}\deg\alpha_i\deg\alpha_j$, and hence the image of $\textrm{P}$ is in the symmetric product $\textrm{Sym}^nH^*(S,\BQ)=H^*(S^{(n)},\BQ)$. The $G$-decomposition on the product of symmetric products $S^{(a_1)}\times\cdots\times S^{(a_n)}$ is defined similarly by using the K\"unneth formula.

Now we turn to the Hilbert scheme $S^{[n]}$. We use the notation $\nu=1^{a_1}\cdots n^{a_n}$ to denote a partition of $n$, where $a_i$ is the number of times that the number $i$ appears in the partition $\nu$. In particular, we have $\sum_{i=1}^nia_i=n$. Let $l(\nu)=a_1+\cdots+a_n$ be the length of the partition $\nu$. For $\nu=1^{a_1}\cdots n^{a_n}$, denote 
\[
S^{(\nu)}=S^{(a_1)}\times\cdots\times S^{(a_n)}.
\]
By \cite[Theorem 1,4]{GS}, there is a canonical decomposition
\[
H^d(S^{[n]},\BQ)=\bigoplus_\nu H^{d-2n+2l(\nu)}(S^{(\nu)},\BQ).
\]

\begin{prop} \label{2.11}
Let $S$ be a smooth quasi-projective surface. Then
\begin{equation} \label{hilb}
G_kH^d(S^{[n]},\BQ)=\bigoplus_\nu G_{k-n+l(\nu)} H^{d-2n+2l(\nu)}(S^{(\nu)},\BQ)
\end{equation}
defines a the $G$-decomposition on $S^{[n]}$.
\end{prop} 

\begin{proof}
The proof is formal. Since $\dim S^{[n]}=2n$ and $\dim S^{(\nu)}=2l(\nu)$, 
\begin{equation}\label{hilb1}
\begin{split}
G_kH^d_c(S^{[n]},\BQ)=&\left(G_{2n-k}H^{4n-d} (S^{[n]},\BQ)\right)^\vee\\
=&\left(\bigoplus_\nu G_{n-k+l(\nu)}H^{2n-d+2l(\nu)} (S^{(\nu)},\BQ)\right)^\vee\\
=&\bigoplus_\nu G_{k-n+l(\nu)}H^{d-2n+2l(\nu)}_c (S^{(\nu)},\BQ),
\end{split}
\end{equation}
where $(-)^\vee$ is the corresponding component in the dual decomposition in the sense of Lemma \ref{du}. Since the decomposition on $S^{(\nu)}$ is a $G$-decomposition, 
\[
\iota:G_{k-n+l(\nu)}H^{d-2n+2l(\nu)}_c (S^{(\nu)},\BQ)\to G_{k-n+l(\nu)}H^{d-2n+2l(\nu)} (S^{(\nu)},\BQ),
\]
we conclude by comparing (\ref{hilb}) and (\ref{hilb1}) that 
\[
\iota: G_kH^d_c(S^{[n]},\BQ)\to G_kH^d(S^{[n]},\BQ),
\]
as desired.
\end{proof}

Proposition \ref{2.11} endows a third grading, the $G$-grading, on the vector space  
\[
\BH=\bigoplus_{n,d,k}\BH_{n,d,k}=\bigoplus_{n,d,k}G_kH^d(S^{[n]},\BQ),
\]
where $n,d,k$ are called the conformal weight, the cohomological degree, and the $G$-degree respectively. We say that a linear operator $A:\BH\to\BH$ is \emph{of degree} $(\Fn,\Fd,\Fk)$ if
\[
A:\BH_{n,d,k}\to\BH_{n+\Fn,d+\Fd,k+\Fk}
\]
for any tridegree $(n,d,k)$.

\section{Tautological classes and multiplicativity of G-decompositions}
\subsection{Nakajima operators and the boundary operator}
In this section we study how Nakajima operators $\Fq_{m}(\alpha)$ and the boundary operator $\partial$ act on $G$-decompositions. We first recall that the coefficients of $t^n$ in the isomorphism
\begin{align*}
\textrm{Sym}^*(H\otimes t\BQ[t])&\cong\BH\\
(\alpha_1t^{n_1})\cdots (\alpha_st^{n_s})&\mapsto \Fq_{n_1}(\alpha_1)\cdots\Fq_{n_s}(\alpha_s)\mathbf{1}.    
\end{align*}
induces an isomorphism
\begin{align}
H^*(S^{[n]},\BQ)&=\bigoplus_{\nu}\bigotimes_{i=1}^n \textrm{Sym}^{a_i} H^*(S,\BQ)[2l(\nu)-2n],\label{4}\\
 \Fq_{n_1}(\alpha_1)\cdots\Fq_{n_s}(\alpha_s)\mathbf{1}&\mapsto \displaystyle\bigotimes_{i=1}^n\textrm{P}\left(\boxtimes_{n_j=i}\alpha_j\right)[2s-2n], \label{5}    
\end{align}
where $\nu=(n_1,\cdots,n_s)=1^{a_1}\cdots n^{a_n}$. By the description of the decomposition (\ref{hilb}), we have the following.

\begin{prop} \label{4.1}
Let $S$ be a smooth quasi-projective surface. Let $G_\bullet H^*(S,\BQ)$ be any $G$-decomposition. Let $(n_1,\cdots,n_s)$ be a partition of $n$, and let $\alpha_i\in G_{k_i}H^{d_i}(S,\BQ)$. Then  
\[
\Fq_{n_1}(\alpha_1)\cdots\Fq_{n_s}(\alpha_s)\mathbf{1}\in G_kH^d(S^{[n]},\BQ),
\]
where $d=\sum(d_i+2n_i-2)=d_1+\cdots+d_s+2n-2s$ and $k=\sum(k_i+n_i-1)=k_1+\cdots+k_s+n-s$.
\end{prop}

\begin{proof}
The cohomological degree $d$ follows directly from the decomposition (\ref{4}) and (\ref{5}). To calculate the $G$-degree, the factor $\textrm{P}\left(\prod_{n_j=i}\alpha_j\right)$ has $G$-degree $\sum_{n_j=i}k_j$ in $H^*(S^{(a_i)},\BQ)$. Therefore by (\ref{hilb}),  
\[
\bigotimes_{i=1}^n\textup{P}\left(\boxtimes_{n_j=i}\alpha_j\right)[2s-2n]
\]
is of $G$-degree 
\[
\sum_{i=1}^n\sum_{n_j=i}k_j+n-l(\nu)=\sum_{j=1}^s k_j+n-s
\]
as desired.
\end{proof}

\begin{prop}\label{4.2}
Let $S$ be a smooth quasi-projective surface equipped with a $G$-decomposition on $H^*(S,\BQ)$. Let $\alpha\in G_kH^d(S,\BQ)$ if $n\ge0$, $\alpha\in G_kH^d_c(S,\BQ)$ if $n<0$. Then $\Fq_n(\alpha)\in \textup{End}_\BQ\BH$ is a linear operator of degree $(n,d+2n-2,k+n-1)$.
\end{prop}

\begin{proof}
The conformal weight and cohomological degree of $\Fq_n(\alpha)$ follows from the definition. We calculate the $G$-degree of $\Fq_n(\alpha)$ from its action of on a linear basis \[\Fq_{m_1}(\beta_1)\cdots\Fq_{m_s}(\beta_s)\mathbf{1},\]
where $m_i$ are positive integers and $\beta_i$ run over a linear basis adapted to the $G$-decomposition $G_\bullet H^*(S,\BQ)$. Denote $\Fg(\beta_i)=k_i$. By Proposition \ref{4.1}, we have  
\[
\Fg(\Fq_{m_1}(\beta_1)\cdots\Fq_{m_s}(\beta_s)\mathbf{1})=\sum_{i=1}^s(k_i+m_i-1).
\]

\begin{enumerate}
\item Case 1: $n>0$. It follows from Proposition \ref{4.1} that 
\[
\Fg\left(\Fq_n(\alpha)\Fq_{m_1}(\beta_1)\cdots \Fq_{m_s}(\beta_s)\mathbf{1}\right)=k+n-1+\sum(k_i+m_i-1).\]
So $\Fq_n(\alpha)$ increases the $G$-degree by $k+n-1$ as desired.
\item Case 2: $n<0$. Then $\Fq_n(\alpha)\mathbf{1}=0$ by degree reason. By Theorem \ref{nak} and Lemma \ref{comm}, we have 
\begin{equation*}
\begin{split}
&\Fq_n(\alpha)\Fq_{m_1}(\beta_1)\cdots \Fq_{m_s}(\beta_s)\mathbf{1}\\
=&\pm\sum_{i=1}^s \Fq_{m_1}(\beta_1)\cdots\Fq_{m_{i-1}}(\beta_{i-1})[\Fq_{n}(\alpha),\Fq_{m_i}(\beta_i)]\Fq_{m_{i+1}}(\beta_{i+1})\cdots\Fq_{m_s}(\beta_s)\mathbf{1}\\
=&\pm\sum_{i=1}^s \Fq_{m_1}(\beta_1)\cdots\Fq_{m_{i-1}}(\beta_{i-1})\delta_{n,-m_i}\int_S{\alpha\beta_i}\Fq_{m_{i+1}}(\beta_{i+1})\cdots\Fq_{m_s}(\beta_s)\mathbf{1}\\
=&\pm\sum_{i=1}^s \delta_{n,-m_i}\int_S{\alpha\beta_i}\cdot\Fq_{m_1}(\beta_1)\cdots\Fq_{m_{i-1}}(\beta_{i-1})\Fq_{m_{i+1}}(\beta_{i+1})\cdots\Fq_{m_s}(\beta_s)\mathbf{1}
\end{split}
\end{equation*}
The constant $\delta_{n,-m_i}\int_S{\alpha\beta_i}$ is nonzero only when $n+m_i=0$ and the Poincar\'e pairing $\int\alpha\beta_i\ne 0$, which is further equivalent to $d+\deg\beta_i=4$ and $k+\Fg(\beta_i)=2$ by Lemma \ref{du}. Thus
\[
\begin{split}
&\Fg\left(\Fq_{m_1}(\beta_1)\cdots\Fq_{m_{i-1}}(\beta_{i-1})\Fq_{m_{i+1}}(\beta_{i+1})\cdots\Fq_{m_s}(\beta_s)\mathbf{1}\right)\\
=&\sum_{j=1}^s(k_j+m_j-1)-(k_i+m_i-1)\\
=&\sum_{j=1}^s(k_j+m_j-1)-(2-k+4-n-1)\\
=&\sum_{j=1}^s(k_j+m_j-1)+(k+n-1).
\end{split}
\]
Therefore all nonzero summands have the same $G$-degree, so does their sum. We conclude that the operator $\Fq_n(\alpha)$ increases the $G$-degree by $k+n-1$.
\end{enumerate}
\end{proof}

\begin{prop} \label{4.3}
Let $S$ be a smooth quasi-projective surface with a $G$-decomposition on $H^*(S,\BQ)$. Let $\alpha\in G_kH^d(S,\BQ)$ if $n\ge0$, and $G_kH^d_c(S,\BQ)$ if $n<0$. 
\begin{enumerate}
    \item $L_n(1)\in\textup{End}_\BQ(\BH)$ is an operator of degree $(n,2n,n)$.
    \item Suppose further that $G_\bullet H^*(S,\BQ)$ is strongly multiplicative. Then $L_n(\alpha)\in \textup{End}_\BQ(\BH)$ is an operator of degree $(n,d+2n,k+n)$.
\end{enumerate}
\end{prop}

\begin{proof}
We check for $n\ge0$. By Remark \ref{G}, there exists basis $\{\beta_i\}$ of $H^*(S,\BQ)$ and $\{\beta^i\}$ of $H^*_c(S,\BQ)$ both adapted to the $G$-decomposition such that $\Fg(\beta_i)+\Fg(\beta^i)=2$ and $\deg\beta_i+\deg\beta^i=4$. By (\ref{Delta+}), we have 
\[
\Delta_+(\alpha)=\sum_i\beta_i\otimes\beta^i\alpha.
\]
By the definition of Virasoro operator,
\[
L_n(\alpha)=\frac{1}{2}\sum_{m\in\BZ}\sum_i\Fq_m(\beta_i)\Fq_{n-m}(\beta^i\alpha).
\]
For simplicity, we denote $d_i=\deg \beta_i$, $d^i=\deg \beta^i$, $k_i=\Fg(\beta_i)$, and $k^i=\Fg(\beta^i)$. We prove (2) first. It follows from strong multiplicativity of the $G$-decomposition and Proposition \ref{4.2} that the operator $\Fq_m(\beta_i)$ is of degree 
\[
(m,d_i+2m-2,k_i+m-1)
\]
and $\Fq_{n-m}(\beta^i\alpha)$
is of degree 
\[
(n-m,d^i+d+2(n-m)-2,k^i+k+(n-m)-1).
\]
Since $d_i+d^i=4$ and $k_i+k^i=2$,  
\[
\Fq_m(\beta_i)\Fq_{n-m}(\beta^i\alpha)
\]
is an operator of degree $(n,d+2n,k+n)$ for all $i$. We conclude that $L_n(\alpha)$ is of degree $(n,d+2n,k+n)$. 

We see from the proof that when $\alpha=1$, the strong multiplicativity is not used. This proves (1). The case $n<0$ is similar.
\end{proof}

\begin{prop} \label{boundary}
Let $S$ be a smooth quasi-projective surface equipped with a $G$-decomposition on $H^*(S,\BQ)$. Let $\partial S^{[n]}$ be the boundary divisor. Then
\[
\partial S^{[n]}\in G_1H^2(S^{[n]},\BQ).
\]
\end{prop}

\begin{proof}
This is \cite[Lemma 2.1]{SZ}.
\end{proof}

\begin{prop} \label{4.5}
Let $S$ be a smooth quasi-projective surface equipped with a strongly multiplicative $G$-decomposition on $H^*(S,\BQ)$. Let $\alpha\in G_kH^d(S,\BQ)$ for $n\ge0$, and $G_kH^d_c(S,\BQ)$ if $n<0$. 
\begin{enumerate}
    \item The operator $(\textnormal{ad}\,\partial)\Fq_1(\alpha)=[\partial,\Fq_1(\alpha)]\in \textup{End}_\BQ(\BH)$ is of degree $(1,d+2,k+1)$.
    \item Suppose further that $K_S\in G_1H^2(S,\BQ)$ where $K$ is the canonical class of $S$. Then the operator $(\textnormal{ad}\,\partial)\Fq_n(\alpha)=[\partial,\Fq_n(\alpha)]\in \textup{End}_\BQ(\BH)$ is of degree $(n,d+2n,k+n)$.
\end{enumerate}
\end{prop}

\begin{proof}
Recall that Proposition \ref{2.4} describes the commutator of boundary and Nakajima operators in terms of Virasoro and Nakajima operators. 
\[
[\partial,\Fq_n(\alpha)]=nL_n(\alpha)+\binom{n}{2}\Fq_n(K\alpha).
\]
\begin{enumerate}
    \item When $n=1$. Proposition \ref{4.3} implies that $[\partial,\Fq_1(\alpha)]=L_1(\alpha)$ is of degree $(1,d+2,k+1)$.
    \item General $n$ with $K\in G_1H^2(S,\BQ)$. Then the strong multiplicativity of $G_\bullet H^*(S,\BQ)$ and Proposition \ref{4.2} implies that $\Fq_n(K\alpha)$ is a linear operator of degree $(n,d+2n,k+n)$. Proposition \ref{4.3} implies that $L_n(\alpha)$ is a linear operator of degree $(n,d+2n,k+n)$. So $(\textrm{ad}\,\partial)\Fq_n(\alpha)=[\partial,\Fq_n(\alpha)]$ is of degree $(n,d+2n,k+n)$.
    \end{enumerate}
\end{proof}

\begin{prop}\label{4.6}
Let $S$ be a smooth quasi-projective surface equipped with a strongly multiplicative $G$-decomposition on $H^*(S,\BQ)$. Suppose further that the canonical class $K\in G_1H^2(S,\BQ)$. Let $\alpha\in G_kH^d(S,\BQ)$ if $n\ge0$ and $G_kH^d_c(S,\BQ)$ if $n<0$. Then the linear operator $(\textnormal{ad}\,\partial)L_n(\alpha):=[\partial,L_n(\alpha)]$ is of degree 
\[
(n,d+2n+2,k+n+1).
\]
\end{prop}

\begin{proof}
We use the notation in Proposition \ref{4.3}. We have $d_i+d^i=4$ and $k_i+k^i=2$. By Lemma \ref{comm}, we have
\begin{equation*}
\begin{split}
[\partial,L_n(\alpha)]
=&\frac{1}{2}[\partial,\sum_{m\in\BZ}\sum_i\Fq_m(\beta_i)\Fq_{n-m}(\beta^i\alpha)]\\
=&\frac{1}{2}\sum_{m\in\BZ}\sum_i[\partial,\Fq_m(\beta_i)\Fq_{n-m}(\beta^i\alpha)]\\
=&\frac{1}{2}\sum_{m\in\BZ}\sum_i[\partial,\Fq_m(\beta_i)]\Fq_{n-m}(\beta^i\alpha)\\
&+\frac{1}{2}\sum_{m\in\BZ}\sum_i\pm\Fq_m(\beta_i)[\partial,\Fq_{n-m}(\beta^i\alpha)].
\end{split}
\end{equation*}
So the degree of linear operator $[\partial,\Fq_m(\beta_i)]\Fq_{n-m}(\beta^i\alpha)$ is
\[
\begin{split}
&(m,d_i+2m,k_i+m)\\
&\textrm{by Proposition }\ref{4.5},\\
+&\left(n-m,d^i+d+2(n-m)-2,k^i+k+(n-m)-1\right)\\
&\textrm{by Proposition }\ref{4.3},\\
=&(n,d+2n+2,k+n+1)\\
&\textrm{because }d_i+d^i=4,\,k_i+k^i=2.
\end{split}
\]
The degree of $\Fq_m(\beta_i)[\partial,\Fq_{n-m}(\beta^i\alpha)]$ is calculated similarly by Proposition \ref{4.3} and Proposition \ref{4.5}. Since all the summands have the same degree, the linear operator $[\partial,L_n(\alpha)]$ is of degree $(n,d+2n+2,k+n+1)$.
\end{proof}

\begin{prop} \label{4.7}
Let $S$ be a smooth quasi-projective surface equipped with a strongly multiplicative $G$-decomposition on $H^*(S,\BQ)$. Suppose further that the canonical class $K\in G_1H^2(S,\BQ)$. Then the boundary operator $\partial$ is of degree $(0,2,1)$.
\end{prop}

\begin{proof}
Since the operator $\partial$ is defined as the cup product with a degree $2$ class, it is obvious to see that its conformal weight is 0 and cohomological degree is 2. To calculate the $G$-degree of $\partial$, it suffices to calculate the action of $\partial$ on a linear basis 
\[
\Fq_{m_1}(\beta_1)\cdots\Fq_{m_s}(\beta_s)\mathbf{1},
\] 
where $m_i$ are positive integers and $\{\beta_i\}$ run over a linear basis adapted to the decomposition $G_\bullet H^*(S,\BQ)$. By Proposition \ref{4.1}, 
\[
\Fg\left(\Fq_{m_1}(\beta_1)\cdots\Fq_{m_s}(\beta_s)\mathbf{1}\right)=\sum_{i=1}^s(\Fg(\beta_i)+m_i-1)
\]
By Lemma \ref{comm} and noting that $\partial\,\mathbf{1}=0$, we have
\[
\begin{split}
&\partial\,\Fq_{m_1}(\beta_1)\cdots\Fq_{m_s}(\beta_s)\mathbf{1}\\
=&\sum_{i=1}^s \pm\Fq_{m_1}(\beta_1)\cdots\Fq_{m_{i-1}}(\beta_{i-1})[\partial,\Fq_{m_i}(\beta_i)]\Fq_{m_{i+1}}(\beta_{i+1})\cdots\Fq_{m_s}(\beta_s)\mathbf{1}.\\
\end{split}
\]
By Proposition \ref{4.3} and Proposition \ref{4.5}, 
\[
\begin{split}
&\Fg(\Fq_{m_1}(\beta_1)\cdots\Fq_{m_{i-1}}(\beta_{i-1})[\partial,\Fq_{m_i}(\beta_i)]\Fq_{m_{i+1}}(\beta_{i+1})\cdots\Fq_{m_s}(\beta_s)\mathbf{1})\\
=&\sum_{j=1}^{i-1}(\Fg(\beta_j)+m_j-1)+(\Fg(\beta_i)+m_i)+\sum_{j=i+1}^s(\Fg(\beta_j)+m_j-1)\\
=&\sum_{j=1}^s(\Fg(\beta_j)+m_j-1)+1
\end{split}
\]
holds for any $1\le i\le s$. Therefore 
\[
\Fg(\partial\,\Fq_{m_1}(\beta_1)\cdots\Fq_{m_s}(\beta_s)\mathbf{1})=\Fg(\Fq_{m_1}(\beta_1)\cdots\Fq_{m_s}(\beta_s)\mathbf{1})+1.
\]
We conclude that the boundary operator $\partial$ is of degree $(0,2,1)$.
\end{proof}

\begin{cor} \label{4.8}
Let $S$ be a smooth quasi-projective surface equipped with a strongly multiplicative $G$-decomposition on $H^*(S,\BQ)$. Suppose further that the canonical class $K\in G_1H^2(S,\BQ)$. Let $\alpha\in G_kH^d(S,\BQ)$ if $n\ge0$ and $G_kH^d_c(S,\BQ)$ if $n<0$. Then the linear operator $(\textup{ad}\,\partial)^m\Fq_n(\alpha)$ is of degree 
\[
(n,d+2n+2m-2,k+n+m-1).
\]
\end{cor}

\begin{proof}
We argue by induction on $m$. The induction base $m=1$ is Proposition \ref{4.5}.(2). Suppose it is proved for $m-1$, i.e. $(\textrm{ad}\,\partial)^{m-1}\Fq_n(\alpha)$ has degree $(n,d+2n+2m-4,k+n+m-2)$. Then by Proposition \ref{4.7}, and induction hypothesis,  
\[
\begin{split}
&(\textrm{ad}\,\partial)^m\Fq_n(\alpha)=[\partial,(\textrm{ad}\,\partial)^{m-1}\Fq_n(\alpha)]\\
=&\partial (\textrm{ad}\,\partial)^{m-1}\Fq_n(\alpha)-(\textrm{ad}\,\partial)^{m-1}\Fq_n(\alpha)\partial.
\end{split}
\]
is a linear operator of degree $(n,d+2n+2m-2,k+n+m-1)$.
\end{proof}

\subsection{Tautological classes}
In this section we will show that for any pure class (see Definition \ref{dfn}) $\alpha\in H^*(S,\BQ)$, the tautological classes $\alpha^{[n]}$ is also pure, and calculate their $G$-degrees. Recall that 
\[
\alpha^{[n]}=\sum_{l\ge 0}\alpha^{[n]}_l,
\]
where $\alpha^{[n]}_l$ is the degree $\deg\alpha+2l-4$ component of $\alpha^{[n]}$ defined in (\ref{taut}). 

\begin{prop}\label{5.1}
Let $S$ is a smooth quasi-projective surface equipped with a strongly multiplicative $G$-decomposition on $H^*(S,\BQ)$. Suppose further that the canonical class $K_S\in G_1H^2(S,\BQ)$. Let $\alpha\in G_kH^d(S,\BQ)$, and $x\in G_KH^D(S^{[n]},\BQ)$. Then 
\[
\alpha^{[n]}_l\cdot x\in G_{K+k+l-2}H^{D+d+2l-4}(S^{[n]},\BQ).
\]
\end{prop}

\begin{proof}
We prove by induction on the lexicographic order of the pair $(n,D)$, the conformal weight and the cohomological degree of $x$. Since 
\[
\BH=\sum_\beta\Fq_1(\beta)\BH+\partial\,\BH,
\]
it suffices to calculate $\alpha^{[n]}_l\cdot \Fq_1(\beta)y$ and $\alpha^{[n]}_l\cdot \partial y$.
\begin{enumerate}
    \item $x=\Fq_1(\beta)y$, $\beta\in G_{k'}H^{d'}(S,\BQ)$. Then $y\in G_{K-k'}H^{D-d'}(S^{[n-1]},\BQ)$.     We have
    \[
    \begin{split}
    \alpha^{[n]}\cdot \Fq_1(\beta)y=&\alpha^{[\bullet]}\Fq_1(\beta)y\\
    =&[\alpha^{[\bullet]},\Fq_1(\beta)]y+\Fq_1(\beta)\alpha^{[\bullet]}y\\
    =&\left(\textrm{exp}(\textrm{ad}\,\partial)\Fq_1(\alpha\beta)\right)y+\Fq_1(\beta)(\alpha^{[n-1]}\cdot y)\\
    &\textrm{by Theorem }\ref{pr},\\
    =&\sum_{m\ge0}\frac{1}{m!}\left((\textrm{ad}\,\partial)^m\Fq_1(\alpha\beta)\right)y+\Fq_1(\beta)(\alpha^{[n-1]}\cdot y).\\
    \end{split}
    \]
    The cohomological degree $D+d+2l-4$ components yield an equation
    \[
    \alpha_l^{[n]}\cdot \Fq_1(\beta)y=\frac{1}{(l-2)!}\left((\textrm{ad}\,\partial)^{l-2}\Fq_1(\alpha\beta)\right)y+\Fq_1(\beta)(\alpha^{[n-1]}_l\cdot y)
    \]
    Both $(\textrm{ad}\,\partial)^{l-2}\Fq_1(\alpha\beta)y$ (by Corollary \ref{4.8}) and $\Fq_1(\beta)(\alpha^{[n-1]}\cdot y)$ (by induction hypothesis and Proposition \ref{4.2}) are of degree 
    \[(n,D+d+2l-4,K+k+l-2).\]
     Therefore
    \[
    \alpha^{[n]}_l\cdot \Fq_1(\beta)y\in G_{K+k+l-2}H^{D+d+2l-4}(S^{[n]},\BQ).
    \]
    \item $x=\partial y$. Then $y\in G_{K-1}H^{D-2}(S^{[n]},\BQ)$. The operator $\partial$ commutes with $\alpha^{[\bullet]}$ because both of them are defined by cup products with classes of even degree. So we have 
    \[
    \alpha^{[n]}\cdot \partial y=\alpha^{[\bullet]}\partial y
    =\partial\alpha^{[\bullet]}y.\]
    The cohomological degree $D+d+2l-4$ components yield an equation
    \[
    \alpha_l^{[n]}\cdot \partial y=\partial (\alpha_l^{[n]}\cdot y)
    \]
    By induction hypothesis $\alpha_k^{[n]}\cdot y$ is of degree
    \[
    (n,D+d+2l-6,K+k+l-3).
    \]
    Since $\partial$ is an operator of degree $(0,2,1)$ (Corollary \ref{4.8}), we have 
    \[
    \alpha_l^{[n]}\cdot \partial y\in G_{K+k+l-2}H^{D+d+2l-4}(S^{[n]},\BQ).
    \]
    \end{enumerate} 
\end{proof}

\begin{prop} \label{5.2}
Let $S$ be a smooth quasi-projective surface equipped with a strongly multiplicative $G$-decomposition on $H^*(S,\BQ)$. Suppose further that the canonical class $K_S\in G_1H^2(S,\BQ)$. Let $\alpha\in G_kH^d(S,\BQ)$. Then the tautological classes 
\[
\alpha^{[n]}_l\in G_{k+l-2}H^{d+2l-4}(S^{[n]},\BQ).
\]
\end{prop}

\begin{proof}
Let $x=1\in G_0H^0(S^{[n]},\BQ)$ in Proposition \ref{5.1}.
\end{proof}

\begin{rmk}
Note that $1\in G_0H^0(S^{[n]},\BQ)$ is of the form $(\Fq_1(1))^n\mathbf{1}$. By iterating Lemma \ref{comm}, we have an explicit formula 
\[
\alpha^{[n]}_l=\frac{1}{(l-2)!}\sum_{i=0}^{n-1}(\Fq_1(1))^i\left((\textrm{ad}\partial)^{l-2} (\Fq_1(\alpha)\right)(\Fq_1(1))^{n-1-i}\mathbf{1}.
\]
\end{rmk}

\subsection{Strong multiplicativity}

\begin{thm} \label{6.1}
Let $n\ge2$. Let $S$ be a smooth quasi-projective surface with a $G$-decomposition on $H^*(S,\BQ)$. If the induced $G$-decomposition $G_\bullet H^*(S^{[n]},\BQ)$ is strongly multiplicative, then the canonical class $K\in G_1H^2(S,\BQ)$. 
\end{thm}

\begin{proof}
If the $G$-decomposition is strongly multiplicative, then Proposition \ref{boundary} implies that the self-intersection of the boundary divisor $\partial S^{[n]}$ should be in $G_2H^4(S^{[n]},\BQ)$. We use two ways to express $\partial S^{[n]}$. On one hand, the linear operator $\partial$ is defined as taking the cup product with $-\frac{1}{2}\partial S^{[n]}$. On the other hand, $\partial S^{[n]}$ can be represented by Nakajima operators as $(\Fq_1(1))^{n-2}\Fq_2(1)\mathbf{1}.$ Therefore, the self intersection of boundary operator is
\[
\begin{split}
\partial S^{[n]}\cdot \partial S^{[n]}=&-2\partial(\Fq_1(1))^{n-2}\Fq_2(1)\mathbf{1}\\
=&-2\sum_{i=0}^{n-3}\pm(\Fq_1(1))^i[\partial, \Fq_1(1)](\Fq_1(1))^{n-3-i}\Fq_2(1)\mathbf{1}\\
&\pm2(\Fq_1(1))^{n-2}[\partial,\Fq_2(1)]\mathbf{1}\\
&\textrm{by Lemma }\ref{comm}\\
=&-2\sum_{i=0}^{n-3}\pm(\Fq_1(1))^iL_1(1)(\Fq_1(1))^{n-3-i}\Fq_2(1)\mathbf{1}\\
&\pm2(\Fq_1(1))^{n-2}(2L_2(1)+\Fq_2(K))\mathbf{1}.\\
&\textrm{by Proposition }\ref{2.4}.
\end{split}
\]
By Proposition \ref{4.2} and Proposition \ref{4.3}.(1), we have
\[
\Fg\left((\Fq_1(1))^iL_1(1)(\Fq_1(1))^{n-3-i}\Fq_2(1)\mathbf{1}\right)=2,
\]
\[
\Fg\left((\Fq_1(1))^{n-2}L_2(1)\mathbf{1}\right)=2,
\]
\[
\Fg\left((\Fq_1(1))^{n-2}\Fq_2(K)\mathbf{1}\right)=1+\Fg(K).
\]
The strong multiplicativity forces that $\Fg(K)=1$, or equivalently, $K\in G_1H^2(S,\BQ)$. 
\end{proof}

\begin{thm} \label{6.2}
Let $n\ge2$. Let $S$ be a smooth quasi-projective surface equipped with a strongly multiplicative $G$-decomposition on $H^*(S,\BQ)$. If $K_S\in G_1H^2(S,\BQ)$, then the $G$-decomposition on $H^*(S^{[n]},\BQ)$ is strongly multiplicative.
\end{thm}

\begin{proof}
We see from Proposition \ref{5.1} and Proposition \ref{5.2} that 
$\Fg(\alpha_l^{[n]}\cdot x)=K+k+l-2$ and $\Fg(\alpha_l^{[n]})=k+l-2$ for any $\alpha\in G_kH^*(S,\BQ)$ and $x\in G_KH^*(S^{[n]},\BQ)$. So 
\[
\Fg(\alpha_l^{[n]}\cdot x)=\Fg(\alpha_l^{[n]})+\Fg(x)
\]
holds for any pure class $\alpha$.
By an induction argument on $t$, we obtain 
\begin{equation} \label{eq}
\Fg\left(\prod_{j=1}^{t}(\alpha_j)_l^{[n]}\cdot x\right)=\Fg\left(\prod_{j=1}^{t}(\alpha_j)_l^{[n]}\right)+\Fg(x)
\end{equation}
for pure classes $\alpha_j$. By \cite{LQW}, $H^*(S^{[n]},\BQ)$ is an $\BQ$-algebra generated by tautological classes $\alpha^{[n]}_l$ where $\alpha$ runs over a linear basis $B$ of $H^*(S,\BQ)$. We may choose $B$ to be adapted to the $G$-decomposition $G_\bullet H^*(S,\BQ)$ since the assignment $\alpha\mapsto\alpha^{[n]}$ is $\BQ$-linear. Therefore, any pure class $z\in G_KH^D(S^{[n]},\BQ)$ can be written as 
\[
z=\sum_{i=1}^s\prod_{j=1}^{t_s}{(\alpha_{ij})}^{[n]}_{l_{ij}}.
\]
such that each summand is in $G_KH^D(S^{[n]},\BQ)$. We conclude from linearity of cup product and (\ref{eq}) that $\Fg(z\cdot x)=\Fg(z)+\Fg(x)$ holds for any pure classes $z$ and $x$. Therefore the $G$-decomposition is strongly multiplicative. 
\end{proof}

The following corollary is a generalization of \cite[Theorem 0.1]{SZ}.

\begin{cor}
Let $S$ be a smooth quasi-projective surface equipped with a strongly multiplicative $G$-decomposition on $H^*(S,\BQ)$. Suppose that the canonical class $K\in G_1H^2(S,\BQ)$. Let $Z_n$ be the universal subscheme in $S\times S^{[n]}$. Then we have
\[
\ch_l(\CO_{Z_n}) \in G_lH^{2l}(S\times S^{[n]},\BQ).
\]
\end{cor}

\begin{proof}
Let $\{\beta_i\}$ be a basis of $H^*(X,\BQ)$ with dual basis $\{\beta^i\}$ of $H^*_c(X,\BQ)$, both adapted to the decomposition $G_\bullet H^*(S,\BQ)$. So
\[
(\beta_i)^{[n]}=p_*(\ch(\CO_{Z_n})\cdot q^*(\beta_i\cdot \textrm{td}(S)))
\]
implies that
\[
\ch(\CO_{Z_n})\cdot q^*\textrm{td}(S)=\sum_i \iota\beta^i\otimes (\beta_i)^{[n]}.
\]
The cohomological degree $2l$ component is   
\[
\begin{split}
\ch_l(\CO_{Z_n})q^*\textrm{td}_0(S)+\ch_{l-1}(\CO_{Z_n})q^*\textrm{td}_1(S)&\\
+\ch_{l-2}(\CO_{Z_n})q^*\textrm{td}_2(S)&=\sum_i \iota\beta^i\otimes (\beta_i)^{[n]}_l. 
\end{split}
\]
By Proposition \ref{5.2}, the $G$-degree of the right side is $2l$. Note that $\textrm{td}_0(S)=1$, $\textrm{td}_1(S)=K_S\in G_1H^2(S,\BQ)$, and $\textrm{td}_2(S)\in H^4(S,\BQ)=G_2H^4(S,\BQ)$. By the K\"unneth property for $G$-decompositions, we have 
\[
q^*\textrm{td}_1(S)\in G_1H^2(S\times S^{[n]},\BQ)
\]
and
\[
q^*\textrm{td}_2(S)\in G_2H^4(S\times S^{[n]},\BQ)
\]
Since the cup product on $S\times S^{[n]}$ are calculated factor-wisely, Theorem \ref{6.1} implies that the $G$-decomposition $G_\bullet H^*(S\times S^{[n]},\BQ)$ is strongly multiplicative. The claim follows from an induction on $l$.
\end{proof}

\section{Applications to perverse decompositions}

\subsection{Perverse filtrations}
We first recall some basic facts about perverse filtrations. Let $D^b_c(Y)$ be the bounded derived category of constructible sheaves. Denote $^\Fp\tau_{\le k}$ the perverse truncation functor. For any object $\CC\in D^b_c(Y)$, we have a canonical morphism
\[
^\Fp\tau_{\le k}\CC\to \CC.
\]
For any proper morphism $f:X\to Y$ between smooth algebraic varieties, we have
\[
^\Fp\tau_{\le k}Rf_*\BQ_X\to Rf_*\BQ_X.
\]
By taking the hypercohomology, there is a natural map 
\begin{equation}\label{p}
\BH^{d-\dim X+r(f)}(X,{^\Fp\tau}_{\le k}Rf_*\BQ_X[\dim X-r(f)])\to H^d(X,\BQ),
\end{equation}
where 
\[
r(f)=\dim X\times_Y X-\dim X
\]
is the defect of semismallness. Define $P_kH^*(X,\BQ)\subset H^*(X,\BQ)$ to be the image of (\ref{p}). By definition, the filtration $P_\bullet H^*(X,\BQ)$ is an increasing filtration, and is called the \emph{perverse filtration} associated with the morphism $f:X\to Y$. The perversity of a class $\alpha\in H^*(X,\BQ)$, denoted as $\Fp^{f}(\alpha)$, is defined to be the number $k$ such that $\alpha\in P_kH^*(X,\BQ)$ and $\alpha\not\in P_{k-1}H^*(X,\BQ)$. The perverse filtration is concentrated in $[0,2r(f)]$, \emph{i.e.} 
\[
0\le \Fp^f(\alpha)\le 2r(f)
\]
for any class $\alpha\in H^*(X,\BQ)$. By the decomposition theorem \cite{BBD}, there is a (non-canonical) decomposition
\[
Rf_*\BQ_X[\dim X-r(f)]=\bigoplus_{i=0}^{2r(f)}\CP_i[-i],
\]
where $P_i$ are perverse sheaves on $Y$. Let $G_iH^*(X,\BQ):=\BH (\CP_i[-i])$, then we have a decomposition
\[
H^*(X,\BQ)=\bigoplus_i G_iH^*(X,\BQ).
\]
Any decomposition obtained this way is called a \emph{perverse decomposition} associated with the morphism $f:X\to Y$. It follows from the definition that it splits the perverse filtration, i.e.
\[
P_kH^*(X,\BQ)=\bigoplus_{i\le k}G_i H^*(X,\BQ).
\]
The following proposition is a generalization of \cite[Proposition 3.1]{Z}. 

\begin{prop} \label{pisg}
Let $f:X\to Y$ be a proper flat morphism between smooth quasi-projective varieties. Then any perverse decomposition associated with $f$ is a $G$-decomposition.
\end{prop}

\begin{proof}
$1\in P_0H^0(X,\BQ)$ follows from the flatness of $f$. Let 
\[
Rf_*\BQ_X[\dim X-r(f)]=\bigoplus_{i=0}^{2r(f)} \CP_{i}[-i]
\]
be a decomposition in $D_c^b(Y)$, where $\CP_i$ are perverse sheaves. We define
$G'_iH^d_c(X,\BQ):=\BH_c(\CP[-i])$. Then 
\[
\iota:G'_iH^d_c(X,\BQ)\to G_iH^d(X,\BQ),
\]
where $\iota:H^*_c(X,\BQ)\to H^*(X,\BQ)$ is the forgetful functor. The self-duality of $Rf_*\BQ_X[\dim X]$ yields isomorphisms $\CP_i\cong\CP_{2r(f)-i)}^\vee$, where $\vee$ is the Verdier dual functor. Therefore, we have the duality equalities
\[
G_iH^d(X,\BQ)\cong G'_{2r(f)-i}H^{2\dim X-i}_c(X,\BQ)^\vee, ~~\forall i\ge0,
\]
which proves that the decompositions $G$ and $G'$ are dual to each other in the sense of Definition \ref{du}.(5).
\end{proof}

The perverse filtration $P_\bullet H^*(X,\BQ)$ associated with a morphism $f:X\to Y$ is called \emph{multiplicative} if 
\[
P_kH^d(X,\BQ)\times P_{k'}H^{d'}(X,\BQ)\xrightarrow{\cup}P_{k+k'}H^{d+d'}(X,\BQ),
\]
or equivalently, $\Fp^f(\alpha\beta)\le\Fp^f(\alpha)+\Fp^f(\beta)$.
For a surface fibered over a curve, the associated perverse filtration is always multiplicative. 
\begin{prop}{\cite[Proposition 4.17]{Z}} \label{pro}
Let $f:S\to C$ be a surjective map from a smooth quasi-projective surface to a smooth curve. Then the perverse filtration associated with $f$ is multiplicative.
\end{prop}

\begin{rmk}
We do not know in general whether every multiplicative perverse filtration admits a strongly multiplicative decomposition to split it. In low dimensional case, perverse decomposition can be constructed explicitly. For example, we will give various equivalent conditions for a Hitchin-type fibration (Definition \ref{Hit}) to admit a strongly multiplicative perverse decomposition. Other examples with the affirmative answer include fibrations $S\to C$ such that $H^1(S,\BQ)=0$ or isotrivial families of curves over curves. 
\end{rmk}

\subsection{Perverse decomposition for Hilbert schemes of fibered surfaces}
Let $f:S\to C$ be a proper surjective morphism from a smooth quasi-projective surface to a smooth  curve. The defect $r(f)=1$, so the perverse filtration associated with $f$ has length 2:
\[
P_0H^*(S,\BQ)\subset P_1H^*(S,\BQ)\subset P_2H^*(S,\BQ)=H^*(S,\BQ).
\]
The fibration $f$ induces a map
\[
\pi_n:S^{[n]}\to C^{(n)},
\]
which is the composition of the Hilbert-Chow morphism $S^{[n]}\to S^{(n)}$ and the induced morphism on the symmetric products $S^{(n)}\to C^{(n)}$. We briefly review the description of the perverse filtration in \cite{Z} and the corresponding perverse decomposition constructed in \cite{SZ}. On the Cartesian product $f^n:S^n\to C^n$, the perverse filtration is
\[
P_kH^*(S^n,\BQ)=\langle \alpha_1\boxtimes\cdots\boxtimes\alpha_n\mid \Fp^f(\alpha_1)+\cdots+\Fp^f(\alpha_n)\le k\rangle.
\]
By taking the $\mathfrak{S}_n$-invariant part, the perverse filtration descends to the ones for the symmetric product $f^{(n)}:S^{(n)}\to C^{(n)}$. 
\[
P_kH^*(S^{(n)},\BQ)=\langle \textrm{P}(\alpha_1\boxtimes\cdots\boxtimes\alpha_n)\mid \Fp^f(\alpha_1)+\cdots+\Fp^f(\alpha_n)\le k\rangle,
\]
where the symmetrization operator $\textrm{P}$ is defined in (\ref{P}). The perverse filtration on the product of symmetric products $S^{(a_1)}\times\cdots\times S^{(a_n)}$ is defined similarly by the K\"unneth formula.

Now we turn to the Hilbert scheme $S^{[n]}$. Recall that for a partition $\nu=1^{a_1}\cdots n^{a_n}$ of $n$, we denote 
\[
S^{(\nu)}=S^{(a_1)}\times\cdots\times S^{(a_n)}.
\]
\begin{thm}{\cite[Corollary 4.14]{Z}}\label{505}
Let $\pi:S\to C$ be a proper map from a smooth quasi-projective surface onto a smooth curve. Then 
\begin{equation} \label{504}
    P_kH^d(S^{[n]},\BQ)=\bigoplus_\nu P_{k-n+l(\nu)}H^{d-2n+2l(\nu)}(S^{(\nu)},\BQ),
\end{equation}
where the perverse filtration is defined by the natural map $h:S^{[n]}\to C^{(n)}$.
\end{thm}

It is straightforward to check that once we fix a strongly multiplicative perverse decomposition $G_\bullet H^*(S,\BQ)$ associated with $f:S\to C$, the $G$-decompositions $G_\bullet H^*(S^n,\BQ)$, $G_\bullet H^*(S^{(n)},\BQ)$, and $G_\bullet H^*(S^{[n]},\BQ)$ constructed in Section 2.2 split the corresponding perverse filtrations. Therefore, they are perverse decompositions associated with maps $f^n:S^n\to C^n$, $f^{(n)}:S^{(n)}\to C^{(n)}$, and $\pi:S^{[n]}\to C^{(n)}$, respectively. Therefore, the main results for $G$-decompositions are valid for perverse decompositions.

\begin{prop}\label{7.5}
Let $f:S\to C$ be a proper surjective morphism from a smooth quasi-projective surface to a smooth curve. Then $G_\bullet H^*(S,\BQ)$ induces a perverse decomposition 
\[
\BH=\bigoplus_{\Fn,\Fd,\Fk} G_\Fk H^\Fd(S^{[\Fn]},\BQ).
\]
Let $\alpha\in G_kH^d(S,\BQ)$ if $n\ge0$, and $G_kH^d_c(S,\BQ)$ if $n<0$. We have
\begin{enumerate}
    \item The Nakajima operator $\Fq_n(\alpha)$ is of degree $(n,d+2n-2,k+n-1)$.
    \item The Virasoro operator $L_n(1)$ is of degree $(n,2n,n)$.
\end{enumerate}
Suppose that $G_\bullet H^*(S,\BQ)$ is strongly multiplicative.
\begin{enumerate}
    \item[(3)] The Virasoro operator $L_n(\alpha)$ is of degree $(n,d+2n,k+n)$.
\end{enumerate}
Suppose further that $K_S\in G_1H^2(S,\BQ)$.
\begin{enumerate}
    \item[(4)] The boundary operator $\partial$ is of degree $(0,2,1)$.
    \item[(5)] The ``cupping with $\alpha^{[n]}_l$" operator is of degree $(0,d+2l-4,k+l-2)$. 
\end{enumerate}
\end{prop}

\begin{proof}
Since the perverse decomposition $G_\bullet H^*(S,\BQ)$ is a $G$-decomposition by Proposition \ref{pisg}, the statements follows from Propositions \ref{4.2}, \ref{4.3}, \ref{4.7}, and \ref{5.1}.
\end{proof}

\begin{prop} \label{7.6}
Let $n\ge2$. Let $f:S\to C$ be a proper surjective morphism from a smooth quasi-projective surface to a smooth curve equipped with a perverse decomposition $G_\bullet H^*(S,\BQ)$. If the induced perverse decomposition $G_\bullet H^*(S^{[n]},\BQ)$ is strongly multiplicative, then the canonical class $K_S\in G_1H^2(S,\BQ)$. If $G_\bullet H^*(S,\BQ)$ is strongly multiplicative, then the converse is true.
\end{prop}

\begin{proof}
The claim follows from Theorem \ref{6.1} and \ref{6.2}.
\end{proof}

Parallel to Proposition \ref{7.6}, we have the following necessary condition for the perverse filtration associated with $f:S^{[n]}\to C^{(n)}$ to be multiplicative.

\begin{prop} \label{pf}
Let $n\ge2$ Let $f:S\to C$ be a proper surjective morphism from a smooth quasi-projective surface to a smooth curve. Let $\pi:S^{[n]}\to C^{(n)}$ be the induced morphism. If the perverse filtration associated with $\pi$ is multiplicative, then $\Fp^f(K_S)\le1$.
\end{prop}

\begin{proof}
Similar to the proof of Theorem \ref{6.1}, we calculate the self-intersection of the boundary divisor $\partial S^{[n]}$. We have seen that
\[
\begin{split}
\partial S^{[n]}\cdot \partial S^{[n]}=&-2\partial(\Fq_1(1))^{n-2}\Fq_2(1)\mathbf{1}\\
=&-2\sum_{i=0}^{n-3}(\Fq_1(1))^iL_1(1)(\Fq_1(1))^{n-3-i}\Fq_2(1)\mathbf{1}\\
&-2(\Fq_1(1))^{n-2}(2L_2(1)+\Fq_2(K_S))\mathbf{1}.
\end{split}
\]
By Proposition \ref{7.5}.(1),(2), we have
\[
(\Fq_1(1))^iL_1(1)(\Fq_1(1))^{n-3-i}\Fq_2(1)\mathbf{1}\in G_2H^4(S^{[n]},\BQ)\subset P_2H^4(S^{[n]},\BQ)
\]
and 
\[
(\Fq_1(1))^{n-2}(2L_2(1))\mathbf{1}\in G_2H^4(S^{[n]},\BQ)\subset P_2H^4(S^{[n]},\BQ).
\]
The multiplicativity of perverse filtration implies that $\Fp^\pi(\partial S^{[n]}\cdot\partial S^{[n]})\le 2$, so
\[
(\Fq_1(1))^{n-2}\Fq_2(K_S)\mathbf{1}\subset P_2H^4(S^{[n]},\BQ).
\]
In the decomposition (\ref{504}), $(\Fq_1(1))^{n-2}\Fq_2(K_S)\mathbf{1}$ is identified with $K_S\boxtimes \textrm{P}(1\boxtimes\cdots\boxtimes1)$ in the summand $\nu=(2,1,\cdots,1)$, Theorem \ref{505} implies that 
\[
\Fp^\pi((\Fq_1(1))^{n-2}\Fq_2(K_S)\mathbf{1})=\Fp^f(K_S)+1.
\]
Therefore $\Fp^f(K_S)\le1$.
\end{proof}

In fact, we have a nice geometric description for $\Fp^{f}(K_S)\le 1$.

\begin{prop} \label{pf2}
Let $f:S\to C$ be a proper surjective map from a quasi-projective surface to a smooth curve. Then $K_S\in P_1H^2(S,\BQ)$ if and only if $f$ is an elliptic fibration.
\end{prop}

\begin{proof}
Let $F$ denote any smooth fiber $f^{-1}(x)$, $x\in C$.

If $f:S\to C$ is an elliptic fibration, then $K_F=0$ and the normal bundle $\CN_{F/S}$ is trivial. Then by adjunction formula
\[
K_S|_F=K_F-F|_F=0.
\]
Note that $\textrm{Gr}_2^PH^2(S,\BQ)$ is spanned by a generic section class, which restricts nontrivially to any general fiber. So $K_S\in P_1H^2(S,\BQ)$.

Conversely, Let $K_S\in P_1H^1(S,\BQ)$. Let $C^\circ$ be the open subset of $C$ consists of points whose fibers are smooth. Let $j:C^\circ\to C$ be the open embedding. By \cite[Theorem 3.2.3]{dCM1}, we have a decomposition
\begin{equation} \label{p2}
Rf_*\BQ_S[1]=\BQ_C[1]\bigoplus(j_*L[1]\oplus\oplus_p\BQ_p^{n_p-1})[-1]\bigoplus(\BQ_C[1])[-2]
\end{equation}
where $p$ runs through singular values of $f$ and $L$ is a local system on $C^\circ$. Then 
\[
K_S\in \BH\left(\BQ_C[1]\bigoplus(j_*L[1]\oplus\oplus_p\BQ_p^{n_p-1})[-1]\right).
\]
Applying the proper base change theorem to the Cartesian square
\[
\begin{tikzcd}
F\arrow[r]\arrow[d,"f'"]& S\arrow[d,"f"]\\
\{x\}\arrow[r]&C,
\end{tikzcd}
\]
the stalk of (\ref{p2}) at $c$ yields
\[
Rf'_*\BQ_{F}=\BQ_x\bigoplus\BQ_x^{2g}[-1]\bigoplus\BQ_x[-2]
\]
where $g$ is the genus of the general fiber. Since $\BH^2\left(\BQ_x\bigoplus\BQ_x^{2g}[-1]\right)=0$, $K_S|_F=0$ under the restriction $H^2(S,\BQ)\to H^2(F,\BQ)$. Since $F$ is a smooth fiber of $f:S\to C$, its normal bundle in $S$ is trivial. Therefore 
\[
K_F=(K_S+F)|_F=0
\]
by the adjunction formula and hence $F$ is an elliptic curve. 
\end{proof}

Combining Proposition \ref{pf}, \ref{pf2} and \cite[Proposition 4.17]{Z}, we have:
\begin{thm} \label{pie}
Let $f:S\to C$ be a proper surjective morphism from a smooth quasi-projective to a smooth curve. Let $\pi_n:S^{[n]}\to C^{(n)}$ be the induced morphism. Then
\begin{enumerate}
    \item The perverse filtration associated with $\pi_1$ is always multiplicative.
    \item Let $n\ge 2$. If the perverse filtration associated with $\pi_n$ is multiplicative, then $f$ is an elliptic fibration.
\end{enumerate}
\end{thm}

\begin{rmk}
It is natural to ask whether the perverse filtration associated with $\pi:S^{[n]}\to C^{(n)}$ is multiplicative if and only if $\Fp^f(K_S)\le 1$. We believe that the statement is true. In fact, by a similar argument as Proposition \ref{5.1}, $\Fp(K_S)\le 1$ implies that 
\[
\alpha^{[n]}_l\cdot x\in P_{K+k+l-2}H^*(S^{[n]},\BQ)
\]
and in particular
\[
\alpha^{[n]}_l\in P_{k+l-2}H^*(S^{[n]},\BQ)
\]
for $\alpha\in P_kH^*(S,\BQ)$ and $x\in P_KH^*(S^{[n]},\BQ)$, but it is difficult to determine the precise perversity of $\alpha^{[n]}_l$, which prevents comparing $\Fp^\pi\left(\alpha^{[n]}_l\right)+\Fp^\pi(x)$ and $\Fp^\pi\left(\alpha^{[n]}_l\cdot x\right)$. However, when $f:S\to\BA^1$ with one singular fiber or when $n=2$, the converse of Theorem \ref{pie} is true. See Section 4.3 and 4.4.
\end{rmk}

\subsection{Hitchin-type fibrations}
In this section, we study the perverse filtration of Hilbert schemes of surfaces fibered over $\BA^1$ which behave like the Hitchin fibrations of 2 dimensional moduli spaces of Higgs bundles. We will show that the perverse filtration is multiplicative if and only if the fibration is elliptic.  More precisely, we consider the fibered surfaces satisfying the following condition:

\begin{defn}\label{Hit}
$f:S\to \BA^1$ called a \emph{Hitchin-type fibration} if it is a proper map from a connected smooth surface onto the affine line, such that the restriction
\[
H^*(S,\BQ)\to H^*(f^{-1}(0),\BQ)
\]
is an isomorphism.
\end{defn}

Let $U$ be the largest open set in $\BA^1$ such that the restriction $f:f^{-1}(U)\to U$ is smooth. Let $j:U\to\BA^1$ be the open embedding. We have the following decomposition for Hitchin-type fibrations. 
\begin{equation}\label{dec}
Rf_*\BQ_S[1]=\BQ_{\BA^1}[1]\bigoplus\left(j_*L[1]\oplus\oplus_p\BQ_0^{\oplus k_p-1}\right)[-1]\bigoplus (\BQ_{\BA^1}[1])[-2],
\end{equation}
where $p$ runs through $\BA^1\setminus U$, $k_p$ is the number of irreducible components of the  fiber $f^{-1}(p)$, and $L$ is the local system on $U$ which corresponds to the representation 
\[
\pi_1(U,t)\to \textrm{GL}(H^1(f^{-1}(t),\BQ)).
\]
for $t\in U$. In particular, we have
\begin{gather}
    H^0(S,\BQ)=P_0H^0(S,\BQ)=\BQ\label{432},\\
    H^1(S,\BQ)=P_1H^1(S,\BQ),\quad P_0H^1(S,\BQ)=0,\\
    P_1H^2(S,\BQ)=\textrm{Gr}_1^PH^2(S,\BQ),\quad P_0H^2(S,\BQ)=0,\\
    \textrm{Gr}^P_2H^2(S,\BQ)\cong\BQ\label{434}.
\end{gather}

To study the existence of a strongly multiplicative perverse decomposition associated with Hitchin-type fibrations, we study the following two maps. Let $cl:H_2(f^{-1}(0),\BQ)\to H^2(f^{-1}(0),\BQ)$ be the composition
\begin{equation}
\begin{tikzcd}\label{long}
H_2(f^{-1}(0),\BQ)\arrow{d}{\cong}[swap]{i_*}\arrow[rr,"cl"]& & H^2(f^{-1}(0),\BQ)\\
H_2(S,\BQ)\arrow{r}{\cong}[swap]{PD}&H^2_c(S,\BQ)\arrow[r,"\iota"] &H^2(S,\BQ)\arrow{u}{i^*}[swap]{\cong}.
\end{tikzcd}
\end{equation}
and let
\[
c:H^1(f^{-1}(0),\BQ)\otimes H^1(f^{-1}(0),\BQ)\to H^2(f^{-1}(0),\BQ)
\]
be the cup product on the central fiber.

\begin{prop} \label{cl}
The image $\textup{Im}\,cl$ is $k-1$ dimensional, where $k$ is the number of irreducible components of $f^{-1}(0)$. More precisely, let 
\[
f^{-1}(0)=b_1 E_1+\cdots+b_kE_k\in H_2(f^{-1}(0),\BQ)
\]
be the cycle theoretic fiber in $S$. Then
\begin{equation} \label{in}
\textup{Im}\, cl=\{c_1p_1+\cdots+c_kp_k\in H^2(f^{-1}(0))\mid b_1c_1+\cdots+b_kc_k=0,\,c_i\in\BQ\},
\end{equation}
where $\{p_i\in H^2(f^{-1}(0),\BQ)\}$ is the dual basis of $\{E_i\}$, i.e.
\begin{eqnarray*}
    p_i:&H_2(f^{-1}(0),\BQ)&\to\BQ\\
         &  E_j&\mapsto \delta_{i,j}. 
\end{eqnarray*}

\end{prop}

\begin{proof}
The properties of the map $cl$ are studied in detail in \cite[Section 2.1]{dCM1} with slightly different notations. For convenience of the reader, we give a self-contained proof here. The map $cl:H_2(f^{-1}(0),\BQ)\to H^2(f^{-1}(0),\BQ)$ induces the refined intersection bilinear form
\[
\eta: H_2(f^{-1}(0),\BQ)\times H_2(f^{-1}(0),\BQ)\to \BQ
\]
with associated symmetric matrix $||E_i\cdot E_j||$. 
By Zariski's lemma \cite[Chapter III, 8.2]{BHPV}, the paring $\eta$ is of rank $k-1$, and hence $cl$ is of rank $k-1$. We have
\begin{equation}\label{211}
cl(E_i)=\sum_{j=1}^k(E_i\cdot E_j)p_j.
\end{equation}
Since $f^{-1}(0)\cdot E_j=0$, 
\begin{equation}\label{212}
\sum_{i=1}^kb_i(E_i\cdot E_j)=0.
\end{equation}
By comparing (\ref{211}) and (\ref{212}), we see that the image of $cl$ is contained in the $(k-1)$-dimensional subspace
\begin{equation} 
\{c_1p_1+\cdots+c_kp_k\in H^2(f^{-1}(0),\BQ)\mid b_1c_1+\cdots+b_kc_k=0,\,c_i\in\BQ\}
\end{equation}
and hence they coincide by dimensional reason.
\end{proof}

\begin{lem}\label{mv}
Let $X=\cup_{i=1}^k E_i$ be a proper singular curve, where $E_i$ are the irreducible components of $X$. Let $\tilde{E}_i$ be the normalization of $E_i$. Then $r:\tilde{X}=\sqcup \tilde{E}_i\to X$ is the normalization of $X$. The restriction map
\[
r^*:H^2(X,\BQ)\xrightarrow{\sim} H^2(\tilde{X},\BQ)=\bigoplus_{i=1}^kH^2(\tilde{E}_i,\BQ)
\]
is an isomorphism and
\[
r^*:H^1(X,\BQ)\twoheadrightarrow H^1(\tilde{X},\BQ)=\bigoplus_{i=1}^kH^1(\tilde{E}_i,\BQ)
\]
is surjective.
\end{lem}

\begin{proof}
By the proof of \cite[Proposition 8.2.19]{Hotta} and \cite[Proposotion 8.2.32]{Hotta},  
\begin{align*}
H^2(X,\BQ)=IH^2(X,\BQ)=IH^2(\tilde{X},\BQ)=H^2(X,\BQ),\\
H^1(X,\BQ)\twoheadrightarrow IH^1(X,\BQ)=IH^1(\tilde{X},\BQ)=H^1(X,\BQ).
\end{align*}
\end{proof}

\begin{prop} \label{snc}
Let $X=\cup_{i=1}^k E_i$ be a proper singular curve, with irreducible components $E_1,\cdots,E_k$. Let $g$ denote the geometric genus, and let $p_i$ be the generator of $H^2(E_i,\BQ)$. Then
\[
\textup{Im}\,c=\left\langle p_i\in H^2(X,\BQ)\mid g(E_i)>0,\,i=1,\cdots,k.\right\rangle
\]
\end{prop}

\begin{proof}
We first show that $\textrm{Im}\,c$ is a linear subspace spanned by a subset of $\{p_1,\cdots,p_k\}$, and then determine which ones are in $\textrm{Im}\,c$.

Denote $r_i^*:H^*(X,\BQ)\to H^*(\tilde{E}_i,\BQ)$. We identify $\gamma\in H^2(X,\BQ)$ with $r^*\gamma\in H^2(\tilde{X},\BQ)$ by Lemma \ref{mv}. In particular, $r^*_i\gamma$ can be viewed in $H^2(X,\BQ)$ and $\gamma=\sum_i r^*_i\gamma.$ We claim that if $\gamma\in H^2(X,\BQ)$ is in $\textrm{Im}\, c$, then its restrictions $r_i^*\gamma$ are also in $\textrm{Im}\, c$ for all $i$. In fact, let $\gamma=\sum_s \alpha_s\cup \beta_s$ where $\alpha_s,\beta_s\in H^1(X,\BQ)$, then $r_i^*\gamma=\sum r_i^*\alpha_s\cup r_i^*\beta_s$. By Lemma \ref{mv}, there exists $\alpha_{is}'\in H^1(X,\BQ)$ such that $r^*\alpha_{is}'=(0,\cdots,r^*_i\alpha_s,\cdots,0)$. We obtain $\beta_{is}'$ similarly. Then $r^*_i\gamma=\sum_s r^*(\alpha_{is}'\cup\beta_{is}')$ is an identity in $H^2(\tilde{X},\BQ)$. Since $r^*_i\gamma\in H^2(\tilde{X})$ is a scalar multiple of $r^*p_i$, we denote $r^*_i\gamma=k\cdot r^*p_i$. Therefore we have $k\cdot r^*p_i=r^*(\sum_s\alpha_{is}'\cup\beta_{is}')$. By Lemma \ref{mv} again, we have $k\cdot p_i=\sum_s\alpha_{is}'\cup\beta_{is}'\in \textrm{Im}\,cl$. We conclude that 
\[
\textrm{Im}\,c=\bigoplus_i r_i^*\textrm{Im}\,c=\langle p_i\mid p_i\in\textrm{Im}\,c\rangle.
\] 

Let $E_i$ be a irreducible component such that the geometric genus $g(E_i)>0$. Then there exists $\alpha_i,\beta_i\in H^1(\tilde{E_i},\BQ)$ such that $\alpha_i\cup\beta_i=[pt\in \tilde{E_i}]=r^*p_i$. By Lemma \ref{mv}, there exists $\alpha,\beta\in H^1(X,\BQ)$ such that $\alpha_i=r^*\alpha$ and $\beta_i=r^*\beta$. Therefore $r^*(\alpha\cup\beta)=r^*p_i$, and hence $p_i=\alpha\cup \beta$ by Lemma \ref{mv}. This implies $p_i\in \mathrm{Im}\,c$.

Similarly, if the geometric genus $g(E_i)=0$, then $H^1(\tilde{E}_i,\BQ)=0$. If $p_i=\sum_s\alpha_{is}\cup\beta_{is}$ were in $\textrm{Im}\,c$, then the right side of the restriction $r_i^*p_i=\sum_sr^*_i\alpha_{is}\cup r^*_i\beta_{is}$ would be zero, which is a contradiction.
\end{proof}

\begin{thm} \label{star}
Let $f:S\to \BA^1$ be a Hitchin-type fibration. Then the following are equivalent.
\begin{enumerate}
    \item $f$ admits a strongly multiplicative perverse decomposition.
    \item $\textup{Im}\, c\cap\textup{Im}\,cl =\{0\}$.
    \item $\dim \textup{Im}\,c\le1$.
    \item $f^{-1}(0)$ has at most 1 irreducible components of positive geometric genus.
\end{enumerate}
\end{thm} 

\begin{proof}
(1)$\Leftrightarrow$(2). By (\ref{432})-(\ref{434}), the existence of a strongly multiplicative decomposition associated with $f$ is equivalent to the existence of a class 
\[\Sigma\in H^2(S,\BQ)\setminus P_1H^2(S,\BQ)\]
such that for any $\alpha,\beta\in H^1(S,\BQ)$, $\alpha\cup\beta=k\Sigma$ for some scalar $k$. Since the restriction map 
\[
r:H^*(S,\BQ)\xrightarrow{\sim}H^*(f^{-1}(0),\BQ)
\]
is an isomorphism, the existence of $\Sigma$ is equivalent to the existence of 
$\sigma\in H^2(f^{-1}(0),\BQ)\setminus r(P_1H^2(S,\BQ))$, such that for any $\alpha,\beta\in H^1(f^{-1}(0),\BQ)$, $\alpha\cup\beta=k\sigma$ for some scalar $k$. Since
\[
P_1H^2(S)=\textrm{Im}\{H^2_c(S)\to H^2(S)\},
\]
(\cite[Proposition 3.1]{Z1}) and $H_2(f^{-1}(0),\BQ)=H^2_c(S,\BQ)$ by (\ref{long}), we have
\[
r(P_1H^2(S,\BQ))=\textrm{Im}\,cl.
\]    
Therefore, the existence of such a $\sigma$ is equivalent to $\dim\textrm{Im}\, c=0$, or $\dim\textrm{Im}\, c=1$ and $\textrm{Im}\, c\not\subset\textrm{Im}\, cl$. Since $\dim \textrm{Im}\,cl=k-1$, they are equivalent to the single condition $\textrm{Im}\, c\cap\textrm{Im}\,cl =\{0\}$, as desired.

(2)$\Rightarrow$(3). Since $\dim \textrm{Im}\,cl=k-1$ by Proposition \ref{cl},  $\textrm{Im}\, c\cap\textrm{Im}\,cl =\{0\}$ implies that $\textrm{Im}\, c\le 1$. 

(3)$\Rightarrow$(2). If $\dim\textrm{Im}\, c=0$, then $\textrm{Im}\, c\cap\textrm{Im}\,cl =\{0\}$. If $\dim\textrm{Im}\, c=1$, then by Proposition \ref{snc}, $\textrm{Im}\, c=\langle p_{i_0}\rangle$ for some $1\le i_0\le k$. By Proposition \ref{cl}, $p_i\not\in \textrm{Im}\,cl$ since $b_i>0$ for all $i$. We conclude that $\textrm{Im}\, c\cap\textrm{Im}\,cl =\{0\}$.

(3)$\Leftrightarrow$(4). This follows from Proposition \ref{snc}.
\end{proof}

A direct corollary is the following.
\begin{cor} \label{eis}
All elliptic Hitchin-type fibrations admit strongly multiplicative perverse decompositions.
\end{cor}

\begin{proof}
Follows from Theorem \ref{star}.(1),(4).
\end{proof}

Combining Proposition \ref{7.6}, Proposition \ref{pf2}, Theorem \ref{pie} and Corollary \ref{eis}, we have the following equivalent condition for the multiplicativity of perverse decomposition for Hilbert schemes of the Hitchin-type fibration.

\begin{thm}\label{4.5*}
Let $n\ge2$. Let $f:S\to\BA^1$ be a Hitchin-type fibration. Let
\[
\pi_n:S^{[n]}\to (\BA^1)^{(n)}=\BA^n
\]
be the induced map. Then the following are equivalent
\begin{enumerate}
    \item $f:S\to \BA^1$ is an elliptic fibration.
    \item $\pi_n$ admits a strongly multiplicative perverse decomposition.
    \item The perverse filtration associated with $\pi_n$ is multiplicative.
    \item The canonical class $K_S\in P_1H^2(S,\BQ)$.
    \item $K_S\in G_1H^2(S,\BQ)$ for any perverse decomposition.
\end{enumerate}
\end{thm}

\begin{proof}
(1)$\Rightarrow$(2). Corollary \ref{eis} implies that $f$ admits a strongly multiplicative perverse decomposition. Proposition \ref{pf2} implies $K_S\in P_1H^2(S,\BQ)$ and hence in $G_1H^2(S,\BQ)$ (since $P_0H^2(S,\BQ)=0$). Then apply Proposition \ref{7.6}. 

(2)$\Rightarrow$(3). This follows from the definition. 

(3)$\Rightarrow$(4). This is Theorem \ref{pie}.(2). 

(4)$\Rightarrow$(1). This is Proposition \ref{pf2}.

(4)$\Leftrightarrow$(5). This follows from $P_0H^2(S,\BQ)=0$.

\end{proof}

\subsection{A direct calculation for $n=2$}
In this section we give a proof of the converse of Theorem \ref{pie} when $n=2$ by a direct calculation via intersection theory. We will see how the canonical class $K_S$ enters the picture from the geometric perspective.

Let $f:S\to C$ be a proper morphism from a smooth quasi-projective surface to a smooth curve. Let $\Delta$ be the diagonal of $S\times S$ and let $E$ be the exceptional divisor in $\textrm{Bl}_\Delta S\times S$. Denote $h:\textrm{Bl}_\Delta S\times S\to C\times C$. We have the following Cartesian diagrams.
\[
\begin{tikzcd}
E\arrow[r,"j"]\arrow[d,"\pi_E"]&\textrm{Bl}_\Delta S\times S\ar[r,"/\mathfrak{S}_2"]\arrow[d,"\pi"]& S^{[2]}\arrow[d]\\
\Delta\arrow[r,"i"]\arrow[d,"f"]&S\times S\arrow[r,"/\mathfrak{S}_2"]\arrow[d,"f\times f"] & S^{(2)}\arrow[d,"f^{(2)}"]\\
C\arrow[r,"i_C"]&C\times C\arrow[r,"/\mathfrak{S}_2"] &C^{(2)}.
\end{tikzcd}
\]

Since the perverse filtration associated with $S^{[2]}\to C^{(2)}$ is the $\BZ/2\BZ$ quotient of the one associated with $h:\textrm{Bl}_\Delta S\times S\to C\times C$, it suffices to show that  $\Fp^f(K_S)\le 1$ implies the multiplicativity of the perverse filtration associated with $h$. To compute the perverse filtration, we first note that 
\[
\BQ_{S\times S}[4]\bigoplus i_*\BQ_\Delta[2]\xrightarrow{\sim} R\pi_*\BQ_{\textrm{Bl}_\Delta S\times S}[4]
\]
is an isomorphism. Pushforward further to $C\times C$, we have
\[
R(f\times f)_*\BQ_{S\times S}[4]\bigoplus i_{C*}Rf_*\BQ_\Delta[2]\xrightarrow{\sim}Rh_*\BQ_{\textrm{Bl}_\Delta S\times S}[4],
\]
So under the isomorphism 
\[
H^*(S\times S,\BQ)\oplus H^*(\Delta,\BQ)[-2]\xrightarrow{(\pi^*,j_*\pi_E^*)} H^*(\textrm{Bl}_\Delta S\times S,\BQ),
\]
the perverse filtrations are identified as 
\begin{equation} \label{f0}
P_k^{f\times f} H^d(S\times S,\BQ)\oplus P_{k-1}^fH^{d-2}(S,\BQ)=P_k^{h}H^d(\textrm{Bl}_\Delta S\times S,\BQ),    
\end{equation}
So $\pi^*$ preserves the perversity and $j_*\pi^*_E$ increases the perversity by 1, \emph{i.e.}
$\Fp^h(\pi^*\alpha)=\Fp^{f\times f}(\alpha)$ and $\Fp^h(j_*\pi^*_E\gamma)=\Fp^f(\gamma)+1$.

Let $\CN=\CN_{\Delta/S\times S}$ be the normal bundle. Then $E=\BP\CN$ is a $\BP^1$-bundle over $\Delta$. Moreover, $\CN\cong T\Delta$. The cup product on $H^*(\textrm{Bl}_\Delta S\times S,\BQ)$ is calculated by
\begin{align}
    \pi^*\alpha\cdot\pi^*\beta&=\pi^*(\alpha\cdot\beta)\label{f1}\\
    \pi^*\alpha\cdot j_*\pi_E^*\gamma&=j_*\pi_E^*(i^*\alpha\cdot\gamma)\label{f2}\\
    j_*\pi_E^*\gamma\cdot j_*\pi_E^*\delta&=-j_*(\pi_E^*\gamma\cdot\pi_E^*\delta\cdot\xi)\label{f3}
\end{align}
where $\xi= c_1(\CO_{\BP\CN}(1))$. We check the multiplicativity for the three types of cup product. \begin{enumerate}
    \item Since the perverse filtration associated with $f$ and $f\times f$ are multiplicative by \cite[Proposition 4.17]{Z} and \cite[Proposition 2.1]{Z}, it follows from (\ref{f0}) and (\ref{f1}) that
\begin{equation}\label{f10}
\begin{split}
\Fp^h(\pi^*(\alpha)\cdot\pi^*(\beta))=\Fp^h(\pi^*(\alpha\cdot\beta))=\Fp^{f\times f}(\alpha\beta)\\\le\Fp^{f\times f}(\alpha)+\Fp^{f\times f}(\beta)=\Fp^{h}(\pi^*\alpha)+\Fp^h(\pi^*\beta).
\end{split}
\end{equation}
\item By (\ref{f0}) and (\ref{f2}), we have 
\begin{equation}\label{f4}
\begin{split}
&\Fp^h(\pi^*(\alpha)\cdot j_*\pi^*_E\gamma)=\Fp^h(j_*\pi_E^*(i^*\alpha\cdot\gamma))\\
&=\Fp^f(i^*\alpha\cdot\gamma)+1\le\Fp^f(i^*\alpha)+\Fp^f(\gamma)+1\\
&\le\Fp^{f\times f}(\alpha)+\Fp^f(\gamma)+1=\Fp^h(\pi^*\alpha)+\Fp^h(j_*\pi^*_E\gamma).
\end{split}
\end{equation}
The second inequality in (\ref{f4}) is because the pullback along the diagonal $i^*(\phi\otimes\psi)=\phi\cdot\psi$ and hence 
\[
\Fp^f(i^*(\phi\otimes\psi))=\Fp^f(\phi\cdot\psi)\le\Fp^f(\phi_i)+\Fp^f(\psi_i)=\Fp^{f\times f}(\phi\otimes\psi)
\]
for all classes $\phi,\psi$.
\item We need to calculate $\Fp^h(j_*(\pi^*_E(\gamma\delta)\cdot\xi))$. Let $\CQ$ be the universal quotient bundle on $X$. Then we have an exact sequence
\[
0\to \CO_{\BP\CN}(-1)\to \pi^*_E\CN\to \CQ\to0.
\]
So $(1-\xi)c(\CQ)=\pi^*_Ec(\CN)$ and hence 
\begin{equation}\label{f5}
    c_1(\CQ)=\pi_E^*c_1(\CN)+\xi=\pi_E^*c_1(T\Delta)+\xi=-\pi_E^*K_\Delta+\xi
\end{equation}
By \cite[Theorem 13.14]{3264}, we have $\pi^*i_*(\gamma\delta)=j_*(c_1(\CQ)\pi^*_E(\gamma\delta))$. So we have $\pi^*i_*(\gamma\delta)=j_*((\xi-\pi^*_EK_\Delta)\pi^*_E(\gamma\delta))$ and hence
\begin{equation}\label{f6}
j_*(\pi^*_E(\gamma\delta)\cdot\xi)=\pi^*i_*(\gamma\delta)+j_*\pi^*_E(\gamma\delta K_\Delta).
\end{equation}
Since $i_*=(\textrm{id},\iota)\circ\Delta_+$, we have 
\begin{equation}\label{f7}
\Fp^h(\pi^*i_*(\gamma\delta))=\Fp^{f\times f}(i_*(\gamma\delta))\le\Fp^{f}(\gamma\delta)+2\le \Fp^{f}(\gamma)+\Fp^{f}(\delta)+2
\end{equation}
by Remark \ref{G} and
\begin{equation}\label{f8}
\Fp^h(j_*\pi^*_E(\gamma\delta K_\Delta))=\Fp^f(\gamma\delta K_\Delta)+1\le \Fp^f(\gamma)+\Fp^f(\delta)+ \Fp^f(K_\Delta)+1.
\end{equation}
If $\Fp^f(K_\Delta)\le1$, then (\ref{f3}), (\ref{f6}), (\ref{f7}) and (\ref{f8}) imply
\begin{equation}\label{f9}
\begin{split}
\Fp^h(j_*\pi_E^*\gamma\cdot j_*\pi_E^*\delta)=\Fp^h(j_*(\pi_E^*(\gamma\delta)\cdot\xi))\\
\le \Fp^{f}(\gamma)+\Fp^{f}(\delta)+2=\Fp^h(j_*\pi_E^*\gamma)+\Fp^h (j_*\pi_E^*\delta).
\end{split}
\end{equation}
\end{enumerate}

Combining Proposition \ref{pf2}, (\ref{f10}), (\ref{f4}) and (\ref{f9}), we have the converse of Theorem \ref{pie} for $n=2$.
\begin{thm} \label{n=2}
Let $f:S\to C$ be a proper surjective map from a smooth quasi-projective surface to a smooth curve. The perverse filtration associated with $\pi_2:S^{[2]}\to C^{(2)}$ is multiplicative if and only if $f$ is an elliptic fibration.
\end{thm}

\end{document}